\documentclass[11pt]{amsart}
\usepackage{wasysym}
\usepackage{rotating}
\usepackage{amsthm}
 \usepackage{amsmath}
\usepackage{graphics}
 \usepackage{amsfonts}
 \usepackage{amssymb}
 \usepackage{amscd}
 \usepackage[all]{xy}
\usepackage{colordvi}
\usepackage{color}
\usepackage{pdfsync}
\usepackage{url}
\usepackage{multirow}
\usepackage{enumerate}

\newtheorem{thm}{Theorem}[section]

\newtheorem{lem}[thm]{Lemma}
\newtheorem{proposition}[thm]{Proposition}
\newtheorem{ex}[thm]{Example}
\newtheorem{conj}[thm]{Conjecture}

\newtheorem{defn}[thm]{Definition}
\newtheorem{rem}[thm]{Remark}
\numberwithin{equation}{section}

\def \cA { {\mathcal A} }
\def \N { {\mathbb N} }
\def \Q { {\mathbb Q} }
\def \R { {\mathbb R} }

\def \Z { {\mathbb Z} }
\def \cP { { P} }
\def \F { {\mathcal F} }
\def \K { {\mathbb K} }
\def \w { {\bf w} }
\def \y { {\bf y} }

\def \V { {\mathbb V} }
\def \cp { {\mathcal P} }
\def \words { [S]^T}
\def \wordsnl { [S]^T_{NL}}
\def \MS {{\mathcal P^*}}

\def\ve#1{\mathchoice{\mbox{\boldmath$\displaystyle\bf#1$}}
{\mbox{\boldmath$\textstyle\bf#1$}}
{\mbox{\boldmath$\scriptstyle\bf#1$}}
{\mbox{\boldmath$\scriptscriptstyle\bf#1$}}}

 \DeclareMathOperator{\diag}{diag}
 \DeclareMathOperator{\conv}{conv}
 \DeclareMathOperator{\cone}{cone}
 
 \newcommand\Side[1]{\begin{sideways}{\small #1}\end{sideways}}
\newcommand\A[1]{\mathcal{A}^{\eqref{model#1}}}
\newcommand\Poly[1]{{P}^{\eqref{model#1}}}
\newcommand\Cone[1]{\mathcal{C}^{\eqref{model#1}}}

\definecolor{darkgreen}{rgb}{0,0.6,0} 

\begin{document}
\title[Degree bounds for a minimal Markov basis for THMC]{Degree bounds for a minimal Markov basis for the three-state toric homogeneous Markov chain model}

\author{David Haws}
\address{Department of Statistics\\
	 University of Kentucky\\
	 861 Patterson Office Tower\\
	 Lexington, KY 40506-0027, USA\\
     Email: david.haws@uky.edu}
\author{Abraham Mart\'in del Campo}
\address{Department of Mathematics\\
         Texas A\&M University\\
         College Station\\
         Texas \ 77843, USA\\
         Email: asanchez@math.tamu.edu}
\author{Ruriko Yoshida}
\address{Department of Statistics\\
	 University of Kentucky\\
	 861 Patterson Office Tower\\
	 Lexington, KY 40506-0027, USA\\
     Email: ruriko.yoshida@uky.edu}

\begin{abstract}
We study the three state toric homogeneous Markov chain model and three special cases of it, namely: (i) when the initial state parameters are constant, (ii) without \emph{self-loops}, and (iii) when both cases are satisfied at the same time. Using as a key tool a directed multigraph associated to the model, the \emph{state-graph}, we give a bound on the number of vertices of the polytope associated to the model which does not depend on the time. Based on our computations, we also conjecture the stabilization of the f-vector of the polytope, analyze the normality of the semigroup, give conjectural bounds on the degree of the Markov bases.
%
%

\end{abstract}

\keywords{Markov bases, time homogeneous Markov chains, polyhedrons, semigroups}

\maketitle

\section{Introduction}

 In this paper, we consider a discrete time Markov chain $X_t$, with
 $t=1,\ldots, T$ ($T\geq 3$), over a finite space of states $[S]=\{1,\ldots,
 S\}$. Let ${\ve w} = (s_1,\ldots,s_T)$ be a path of length $T$ on states
 $[S]$, which is sometimes  written as $\w = (s_1\cdots s_T)$ or simply $\w =
 s_1 \cdots s_T$.  We are interested in Markov bases of toric ideals arising
 from the following statistical models

\begin{equation}\label{thmc}
p(\w) = c\gamma_{s_1}\beta_{s_1,s_2}\cdots \beta_{s_{T-1},s_T}.
\end{equation}

where $c$ is a normalizing constant, $\gamma_{s_i}$ indicates the probability
of the initial state, and $\beta_{s_i, s_j}$ are the transition probabilities
from state $s_i$ to $s_j$. The model \eqref{thmc} is called a toric homogeneous
Markov chain (THMC) model.

Commonly in practice, it is important to consider the case where the initial
parameters are constant; this is, when
$\gamma_{1}=\gamma_{2}=\cdots=\gamma_{S}$; we refer to this case as the THMC
model without initial parameters.  Another simplification that arise from
practice is when we consider only the transition probability between two
different states, i.e. when $\beta_{s_i,s_j}=0$ whenever $s_i=s_j$; this
situation is called a THMC model without self-loops.

In order to simply the notation throughout
this paper we refer to them as Model \eqref{model1}, Model
\eqref{model2}, Model \eqref{model4}, and Model
\eqref{model3}, according to the following:

\begin{enumerate}[{\bf (a):}] 
  \item\label{model1} THMC model \eqref{thmc} 
  \item\label{model2} THMC model without initial parameters: when $\gamma_1=\cdots=\gamma_S$ 
  \item\label{model3} THMC model without self-loops: $\beta_{s_i,s_j}=0$ whenever $s_i=s_j$.
  \item\label{model4} THMC model without initial parameters and without self-loops, i.e., both (b) and (c) are satisfied
\end{enumerate}

%


In 2010, Hara and Takemura\cite{Hara:2010vn} gave a complete description of a Markov basis for Model \eqref{model1}, when $[S] =\{1,2\}$ and $T$ is arbitray, and also for the case when $T = 3$ and $[S]$ is arbitrary. In their next paper\cite{Hara:2010uq}, the authors provided a Markov basis for Model \eqref{model2}, when $[S] = \{1,2\}$, $T$ arbitrary. In these articles, all moves found were of degree four or less, regardless of the value $T$. Motivated by these results, we studied Markov bases of Models \eqref{model1} -- \eqref{model4}. Specifically, we are interested in showing
that the degree of a minimum Markov basis is bounded when $S$ is fixed and $T$ is arbitrary.  Each model has an associated \emph{design matrix} (defined in Section
\ref{sec:notation}) which translates observed data $\ve w = (s1,\ldots,s_T)$
into the sufficient statistic. The sufficient statistic are the number of
transitions from states $i$ to $j$, for all $i,j \in [S]$, and for Model \eqref{model1} and \eqref{model3}, it also includes the initial state.  This paper is organized as follows; in Section
\ref{sec:notation} we describe the design matrices for the above models and
introduce the state graph, a useful tool we use throughout the paper.  In Section \ref{sec:normality} that the semigroups generated by the columns of the design matrices for Model \eqref{model1} and Model \eqref{model2} are not
normal. We also provide computational evidence that for  Model \eqref{model3} and Model
\eqref{model4} the corresponding semigroups are normal, and conjecture that this holds in general. 

In Section \ref{sec:SNF} we study the properties of the Smith normal form of the design matrix and we use some of these results in Section \ref{sec:polytope}, to show the following for Model \eqref{model4}.
\begin{thm}
Let $S=3$. The number of vertices of $\Poly{4}$ is bounded by some constant $C$
which does not depend on $T$. 
\end{thm}

Given the above theorem and our normality conjecture for Model \eqref{model4},
one can prove the following conjecture: 
\begin{conj}\label{conj:degreeboundS3}
We consider Model \eqref{model4}. Then 
for $S = 3$ and for any $T \geq 4$, a minimum Markov basis for the toric ideal
$I_{\A{4}}$ consists of binomials of degree less than or equal to
$d= 6$.  Moreover, there are only finitely many moves up to a
certain shift equivalence relation. 
\end{conj}

Additionally for Model \eqref{model4}, we present in Section \ref{sec:computations} some of our experimental results that suggest that the $f$-vector of the polytope defined by the design matrix stabilized (periodically) indepently of $T$. Our results also suggest that for $S\geq 3$, the bound $d$ of Conjecture \ref{conj:degreeboundS3} depend linearly on $S$; we present these conjectures formally in Section \ref{sec:openproblems}.

\section{Notation}
\label{sec:notation}
Let $\words$ be the set of all words of length $T$ on states $[S]$. Similarly
let $\wordsnl$ be the set of all words of length $T$ on states $[S]$ such that
every word has no self-loops; that is, if $w=(s_1,\ldots,s_T) \in \wordsnl$
then $s_i \neq s_{i+1}$ for $i = 1,\ldots,T-1$. We define $\MS(\words)$ to be the set of
all multisets of words in $\words$. Similarly, we define $\MS(\wordsnl)$ to be
the set of all multisets of words in $\wordsnl$.

Let $\V\left(\words\right)$ be the real vector
space with free basis $\words$ and similarly let $\V\left(\wordsnl\right)$ be the real vector
space with free basis $\wordsnl$. Note that 
$\V\left(\words\right) \cong \R^{S^T}$ and 
$\V\left(\wordsnl\right) \cong \R^{S(S-1)^T}$. We recall some definitions from the classic paper of Pachter and Sturmfels\cite{Pachter:2005kx}. Let $A = (a_{ij})$ be a
non-negative integer $d \times m$ matrix with the property that all column sums
are equal:
\begin{equation*}
\sum_{i=1}^d a_{i1} = \sum_{i=1}^d a_{i2} = \cdots = \sum_{i=1}^d a_{im}. 
\end{equation*}
Let $A = [a_1 \; a_2 \; \cdots \; a_m]$ where $a_j$ are the column vectors of
$A$ and define $\theta^{a_j} = \prod_{i=1}^d \theta_1^{a_{ij}}$ for $j=
1,\ldots,m$. The \emph{toric model} of $A$ is the image of the orthant
$\R^d_{\geq 0}$ under the map 
\begin{equation*}
f: \R^d \rightarrow \R^m, \quad \theta \mapsto \frac{1}{\sum_{j=1}^m \theta^{a_1}}\left( \theta^{a_1}, \ldots, \theta^{a_m} \right).
\end{equation*}
Here we have $d$ parameters $\theta = (\theta_1,\ldots,\theta_d)$ and a
discrete state space of size $m$. In general, the discrete space will be the
set of all possible words on $[S]$ of length $T$ and we can think of
$\theta_1,\ldots, \theta_d$ as the logarithm of the probabilities
$\gamma_{s_1},\beta_{s_1,s_2}\cdots \beta_{s_{T-1},s_T}$. Below we specify
this relation for Models \eqref{model1}, \eqref{model2}, \eqref{model3},
\eqref{model4}. 


\subsection{Model \eqref{model1}}
Consider the state space $\words$. Model \eqref{model1} is parametrized by $\gamma_1,\ldots,\gamma_S$ and
$\beta_{11},\beta_{12},\ldots,\beta_{SS}$; thus, it is parametrized by all
positive real vectors of length $S$ and all positive $S \times S$ real
matrices.  Thus, the number of parameters is
$d=S + S^2$ and the number of transition of states is $m = S^T$. This toric model is represented by the $(S + S^2) \times S^T$ matrix $\A{1}_{S,T}$, whose rows are indexed by $[S] \cup [S]^2$ and the columns are indexed by words $\words$, and it is defined as follows.
\begin{enumerate}
    \item The entry of $\A{1}_{S,T}$ indexed by row $s \in [S]$ and
    column $w=(s_1,\ldots,s_T) \in \words$ is $1$ if $ s = s_1$ and $0$
    else.
    \item In row $\sigma_1 \sigma_2 \in
    [S]^2$ and column $w=(s_1,\ldots,s_T) \in \words$, the entry of $\A{1}_{S,T}$ is equal to $|\left\{\,
    i \in \{1,\ldots,T-1\} \mid \sigma_1 \sigma_2 = s_i s_{i+1} \, \right\}|$.
\end{enumerate}
\begin{ex}
Ordering $[S] \cup [S]^2$ and $\words$ lexicographically, the matrix $\A{1}_{2,4}$ is:
\begin{center}
\begin{tabular}{c| c c c c c c c c c c c c c c c c}
& \Side{\scriptsize 1111 }&\Side{\scriptsize1112 }&\Side{\scriptsize1121 }&\Side{\scriptsize1122 }&\Side{\scriptsize1211 }&\Side{\scriptsize1212 }&\Side{\scriptsize1221 }&\Side{\scriptsize1222 }&\Side{\scriptsize2111 }&\Side{\scriptsize2112 }&\Side{\scriptsize2121 }&\Side{\scriptsize2122 }&\Side{\scriptsize2211 }&\Side{\scriptsize2212 }&\Side{\scriptsize2221 }&\Side{\scriptsize2222}\\
\hline
$1$  & 1 & 1 & 1 & 1 & 1 & 1 & 1 & 1 & 0 & 0 & 0 & 0 & 0 & 0 & 0 & 0\\  
$2$  & 0 & 0 & 0 & 0 & 0 & 0 & 0 & 0 & 1 & 1 & 1 & 1 & 1 & 1 & 1 & 1\\  
$11$ & 3 & 2 & 1 & 1 & 1 & 0 & 0 & 0 & 2 & 1 & 0 & 0 & 1 & 0 & 0 & 0 \\
$12$ &0 & 1 & 1 & 1 & 1 & 2 & 1 & 1 & 0 & 1 & 1 & 1 & 0 & 1 & 0 & 0\\
$21$ &0 & 0 & 1 & 0 & 1 & 1 & 1 & 0 & 1 & 1 & 2 & 1 & 1 & 1 & 1 & 0\\
$22$ &0 & 0 & 0 & 1 & 0 & 0 & 1 & 2 & 0 & 0 & 0 & 1 & 1 & 1 & 2 & 3
\end{tabular}
\end{center} 
\end{ex}

\subsection{Model \eqref{model2}}
Similarly, Model \eqref{model2} is parametrized by all
positive $S \times S$ real matrices, as it is parametrized by
$\beta_{11},\beta_{12},\ldots,\beta_{SS}$. Thus, the number of parameters is
$d=S^2$ and the number of transitions is $m = S^T$. Model \eqref{model2} is represented by the $S^2 \times S^T$ matrix $\A{2}_{S,T}$ whose rows are indexed by $[S]^2$ and the columns are
indexed by words in $\words$.  The entry of $\A{2}_{S,T}$ indexed by row
$\sigma_1 \sigma_2 \in [S]^2$ and column $w=(s_1,\ldots,s_T) \in \words$ is
equal to $|\left\{\, i \in \{1,\ldots,T-1\} \mid \sigma_1 \sigma_2 = s_i
s_{i+1} \, \right\}|$.
\begin{ex}
Ordering $[S]^2$ and $\words$ lexicographically, 
the matrix $\A{2}_{2,4}$ is:
\begin{center}
\begin{tabular}{c|c c c c c c c c c c c c c c c c}
& \Side{\scriptsize 1111 }&\Side{\scriptsize1112 }&\Side{\scriptsize1121 }&\Side{\scriptsize1122 }&\Side{\scriptsize1211 }&\Side{\scriptsize1212 }&\Side{\scriptsize1221 }&\Side{\scriptsize1222 }&\Side{\scriptsize2111 }&\Side{\scriptsize2112 }&\Side{\scriptsize2121 }&\Side{\scriptsize2122 }&\Side{\scriptsize2211 }&\Side{\scriptsize2212 }&\Side{\scriptsize2221 }&\Side{\scriptsize2222}\\
\hline
$11$ & 3 & 2 & 1 & 1 & 1 & 0 & 0 & 0 & 2 & 1 & 0 & 0 & 1 & 0 & 0 & 0 \\
$12$ &0 & 1 & 1 & 1 & 1 & 2 & 1 & 1 & 0 & 1 & 1 & 1 & 0 & 1 & 0 & 0\\
$21$ &0 & 0 & 1 & 0 & 1 & 1 & 1 & 0 & 1 & 1 & 2 & 1 & 1 & 1 & 1 & 0\\
$22$ &0 & 0 & 0 & 1 & 0 & 0 & 1 & 2 & 0 & 0 & 0 & 1 & 1 & 1 & 2 & 3
\end{tabular}
\end{center} 
\end{ex}

\subsection{Model \eqref{model3}}
For Model \eqref{model3}, we consider the state space $\wordsnl$. This model is parametrized by the positive real variables given by $\gamma_1,\ldots,\gamma_S$ and $\beta_{12},\, \beta_{13},\, \ldots,\beta_{1,S}, \, \beta_{21},$ $ \beta_{23},\, \ldots, \, \beta_{S-1,S}$. Thus,
the number of parameters is $d=S^2$ and the number of transitions of state is $m = \left|
\wordsnl \right|$. Model \eqref{model3} is the toric model represented by the
$S^2 \times S(S-1)^{T-1}$ matrix $\A{3}_{S,T}$ defined below. The rows of
$\A{3}_{S,T}$ are indexed by $[S] \cup [S]^2 \setminus \left\{\, (i,i) \mid i \in [S]\,\right\}$ and the columns are indexed by
words $\wordsnl$.
\begin{enumerate}
    \item The entry of $\A{3}_{S,T}$ indexed by row $s \in [S]$ and
    column $w=(s_1,\ldots,s_T) \in \wordsnl$ is $1$ if $ s = s_1$ and $0$
    else.
    \item The entry of $\A{3}_{S,T}$ indexed by row $\sigma_1\sigma_2 \in [S]^2$, where $\sigma_1 \neq \sigma_2$, and column $w=(s_1,\ldots,s_T) \in \wordsnl$ is equal to $|\left\{\,
    i \in \{1,\ldots,T-1\} \mid \sigma_1 \sigma_2 = s_i s_{i+1} \, \right\}|$.
\end{enumerate}
\begin{ex}
For $S=3$ and $T=4$, after ordering $[S] \cup [S]^2 \setminus \left\{\, (i,i) \mid i \in [S]\,\right\}$
and $\wordsnl$ lexicographically, the matrix
$\A{3}_{3,4}$ is:
\begin{center}
\begin{tabular}{c|ccccccccccccc}
& \Side{\scriptsize 1212 } & \Side{\scriptsize 1213 } & \Side{\scriptsize 1231 } & \Side{\scriptsize 1232 } & \Side{\scriptsize 1312 } & \Side{\scriptsize 1313 } & \Side{\scriptsize 1321 } & \Side{\scriptsize 1323 } & \Side{\scriptsize 2121 } & \Side{\scriptsize 2123 } & \Side{\scriptsize 2131 } & \Side{\scriptsize 2132 }  & $\cdots$ \\
        \hline
1  & 1 & 1 & 1 & 1 & 1 & 1 & 1 & 1 & 0 & 0 & 0 & 0  & $\cdots$ \\
2  & 0 & 0 & 0 & 0 & 0 & 0 & 0 & 0 & 1 & 1 & 1 & 1  & $\cdots$ \\
3  & 0 & 0 & 0 & 0 & 0 & 0 & 0 & 0 & 0 & 0 & 0 & 0  & $\cdots$ \\
12 & 2 & 1 & 1 & 1 & 1 & 0 & 0 & 0 & 1 & 1 & 1 & 0  & $\cdots$ \\
13 & 0 & 1 & 1 & 0 & 0 & 2 & 1 & 1 & 0 & 0 & 0 & 1  & $\cdots$ \\
21 & 1 & 1 & 0 & 0 & 0 & 0 & 1 & 0 & 2 & 1 & 0 & 1  & $\cdots$ \\
23 & 0 & 0 & 0 & 1 & 1 & 0 & 0 & 1 & 0 & 1 & 1 & 0  & $\cdots$ \\
31 & 0 & 0 & 1 & 1 & 0 & 1 & 0 & 0 & 0 & 0 & 1 & 1  & $\cdots$ \\
32 & 0 & 0 & 0 & 0 & 1 & 0 & 1 & 1 & 0 & 0 & 0 & 0  & $\cdots$ 
\end{tabular}\\
$\phantom{A}$
\begin{tabular}{ccccccccccccc|c}
$\cdots$ & \Side{\scriptsize 2312 } & \Side{\scriptsize 2313 } & \Side{\scriptsize 2321 } & \Side{\scriptsize 2323 } & \Side{\scriptsize 3121 } & \Side{\scriptsize 3123 } & \Side{\scriptsize 3131 } & \Side{\scriptsize 3132 } & \Side{\scriptsize 3212 } & \Side{\scriptsize 3213 } & \Side{\scriptsize 3231 } & \Side{\scriptsize 3232 } \\
        \hline
$\cdots$  & 0 & 0 & 0 & 0 & 0 & 0 & 0 & 0 & 0 & 0 & 0 & 0 & 1 \\
$\cdots$  & 1 & 1 & 1 & 1 & 0 & 0 & 0 & 0 & 0 & 0 & 0 & 0 & 2 \\
$\cdots$  & 0 & 0 & 0 & 0 & 1 & 1 & 1 & 1 & 1 & 1 & 1 & 1 & 3 \\
$\cdots$  & 0 & 0 & 0 & 0 & 1 & 1 & 1 & 0 & 0 & 0 & 0 & 0 & 12\\
$\cdots$  & 1 & 1 & 0 & 0 & 0 & 0 & 0 & 1 & 1 & 1 & 0 & 0 & 13\\
$\cdots$  & 1 & 0 & 1 & 0 & 1 & 1 & 0 & 1 & 0 & 0 & 0 & 0 & 21\\
$\cdots$  & 0 & 1 & 1 & 2 & 0 & 0 & 1 & 0 & 0 & 0 & 1 & 1 & 23\\
$\cdots$  & 0 & 1 & 0 & 0 & 1 & 0 & 1 & 0 & 2 & 1 & 1 & 0 & 31\\
$\cdots$  & 1 & 0 & 1 & 1 & 0 & 1 & 0 & 1 & 0 & 1 & 1 & 2 & 32
\end{tabular}

\end{center} 
\end{ex}
\subsection{Model \eqref{model4}}
Lastly, Model \eqref{model4} is parametrized by
$S(S-1)$ positive real variables. That is, it is parametrized by
$\beta_{12},\beta_{12},\ldots,\beta_{1,S},\beta_{21},\beta_{23},\ldots,\beta_{S-1,S}$.
Thus, the number of parameters is $d=S(S-1)$ and the number of transitions is $m =
\left| \wordsnl \right|$. Model \eqref{model4} is the toric model represented
by the $S(S-1) \times S(S-1)^{T-1}$ matrix $\A{4}_{S,T}$ whose rows are indexed by $[S]^2 \setminus \left\{\,
(i,i) \mid i \in [S]\,\right\}$ and the columns are indexed by words
$\wordsnl$.  The entry of $\A{4}_{S,T}$ indexed by row $\sigma_1\sigma_2 \in
[S]^2$, where $\sigma_1 \neq \sigma_2$, and column $w=(s_1,\ldots,s_T) \in
\wordsnl$ is equal to $|\left\{\, i \in \{1,\ldots,T-1\} \mid \sigma_1 \sigma_2
= s_i s_{i+1} \, \right\}|$.
\begin{ex}
Ordering $[S]^2 \setminus \left\{\, (i,i) \mid i \in [S]\,\right\}$ and
$\wordsnl$ lexicographically and letting $S=3$ and $T=4$, the matrix
$\A{4}_{3,4}$ is:
\begin{center}
\begin{tabular}{c|ccccccccccccc}
& \Side{\scriptsize 1212 } & \Side{\scriptsize 1213 } & \Side{\scriptsize 1231 } & \Side{\scriptsize 1232 } & \Side{\scriptsize 1312 } & \Side{\scriptsize 1313 } & \Side{\scriptsize 1321 } & \Side{\scriptsize 1323 } & \Side{\scriptsize 2121 } & \Side{\scriptsize 2123 } & \Side{\scriptsize 2131 } & \Side{\scriptsize 2132 } & $\cdots$ \\
\hline
12 & 2 & 1 & 1 & 1 & 1 & 0 & 0 & 0 & 1 & 1 & 1 & 0 & $\cdots$ \\
13 & 0 & 1 & 1 & 0 & 0 & 2 & 1 & 1 & 0 & 0 & 0 & 1 & $\cdots$ \\
21 & 1 & 1 & 0 & 0 & 0 & 0 & 1 & 0 & 2 & 1 & 0 & 1 & $\cdots$ \\
23 & 0 & 0 & 0 & 1 & 1 & 0 & 0 & 1 & 0 & 1 & 1 & 0 & $\cdots$ \\
31 & 0 & 0 & 1 & 1 & 0 & 1 & 0 & 0 & 0 & 0 & 1 & 1 & $\cdots$ \\
32 & 0 & 0 & 0 & 0 & 1 & 0 & 1 & 1 & 0 & 0 & 0 & 0 & $\cdots$ 
\end{tabular}\\
$\phantom{A}$
\begin{tabular}{ccccccccccccc|c}
$\cdots$ & \Side{\scriptsize 2312 } & \Side{\scriptsize 2313 } & \Side{\scriptsize 2321 } & \Side{\scriptsize 2323 } & \Side{\scriptsize 3121 } & \Side{\scriptsize 3123 } & \Side{\scriptsize 3131 } & \Side{\scriptsize 3132 } & \Side{\scriptsize 3212 } & \Side{\scriptsize 3213 } & \Side{\scriptsize 3231 } & \Side{\scriptsize 3232 } \\
\hline
$\cdots$ & 0 & 0 & 0 & 0 & 1 & 1 & 1 & 0 & 0 & 0 & 0 & 0 & 12 \\
$\cdots$ & 1 & 1 & 0 & 0 & 0 & 0 & 0 & 1 & 1 & 1 & 0 & 0 & 13 \\
$\cdots$ & 1 & 0 & 1 & 0 & 1 & 1 & 0 & 1 & 0 & 0 & 0 & 0 & 21 \\
$\cdots$ & 0 & 1 & 1 & 2 & 0 & 0 & 1 & 0 & 0 & 0 & 1 & 1 & 23 \\
$\cdots$ & 0 & 1 & 0 & 0 & 1 & 0 & 1 & 0 & 2 & 1 & 1 & 0 & 31 \\
$\cdots$ & 1 & 0 & 1 & 1 & 0 & 1 & 0 & 1 & 0 & 1 & 1 & 2 & 32
\end{tabular}

\end{center} 
\end{ex}

\subsection{Sufficient statistics, ideals, and Markov basis}

We refer to the matrices $\A{1}$, $\A{2}$,$\A{3}$, and $\A{4}$ as \emph{design
matrices} throughout this paper.
Let $\mathcal A$ be $\A{1}$ or $\A{2}$, for $w \in \words$ we denote the column of $\mathcal A$ indexed by $w$ by $\mathcal
A_{S,T}(w)$ or simply by $\mathcal A_w$ when $S$ and $T$ are understood. Thus, by extending linearly, the map $\mathcal A: \V(\words) \rightarrow \R^l$ is well-defined,
where $l$ is either $S^2$ or $S(S-1)$ respectively, depending on the model.

Similarly let $w \in \wordsnl$, and $\mathcal A$ be $\A{3}$ or $\A{4}$. We
let $\mathcal A_{S,T}(w)$ or $\mathcal A_w$ denote the column of $\mathcal A$ indexed by $w$. Again,
we extend the map linearly so that $\mathcal A: \V(\wordsnl) \rightarrow \R^l$
is well-defined, where $l$ is either $S^2$ or $S(S-1)$ depending on the model.

Let $W = \{ w_1,\ldots,w_N \} \in \MS(\words)$ ($\in \MS(\wordsnl$ resp.) and we 
regard $W$ as observed data which can be summarized in the \emph{data vector}
$\ve u \in \N^{S^T}$ ($\ve u \in \N^{S(S-1)^{T-1}}$ resp.)
We index $\ve u$ by words in $\words$ ( or $\wordsnl$ resp.), and $u_{w}$ is equal to the
number of words in $W$ equal to $w$. Let $\mathcal A$ be one of the design matrices $\A{1},\, \A{2},\, \A{3},$ or $\A{4}$; then, as $\mathcal A$ is linear, then $\mathcal A(\ve u)$ is well-defined. We also adopt the notation, $\mathcal A(W) := \mathcal A(\ve u)$.  For $W$ from $\MS(\words)$ or $\MS(\wordsnl)$ (depending on the model) with data vector $\ve u$, then $\cA \ve u$ is the \emph{sufficient statistic} for the corresponding model.  Often the data
vector $\ve u$ is also referred to as a \emph{contingency table} and $\cA \ve u$ are
referred to as the \emph{marginals}. (For a proof of sufficient statistics for
Model \eqref{model1} see Hara and Takemura\cite{Hara:2010vn} and a proof of sufficient
statistics for Model \eqref{model2} see Hara and Takemura\cite{Hara:2010uq}.)

The design matrices $\A{1}$ and $\A{2}$ above define two toric ideals which are of our interest, as their set of generators are in bijection with the Markov bases of these models. For the Model \eqref{model2}, let $I_{\A{2}}$ be the toric ideal defined as the kernel of the homomorphism of polynomial rings $\psi:\K[\{\,p(\w) \mid \w \in \words\,\}] \rightarrow \K[\{\,\beta_{ij} \mid i,j \in
[S]\,\}]$ defined by $\psi: p(\w) \mapsto \beta_{s_1s_2}\cdots
\beta_{s_{T-1}s_T}$, where $\K$ is a field and $\{\,p(\w) \mid \w \in \words\,\}$ is
regarded as a set of indeterminates. Similarly, the toric ideal $I_{\A{1}}$ corresponding to Model \eqref{model1} is defined as
the kernel of the homomorphism $\psi': \K[\{\,p(\w) \mid \w
\in \words\,\}] \rightarrow \K[\{\,\beta_{ij} \mid i,j \in
[S]\,\}][\{\,\gamma_k \mid k \in [S]\,\}]$ defined by $\psi': p(\w)
\mapsto \gamma_{s_1}\beta_{s_1s_2}\cdots \beta_{s_{T-1}s_T}$, where again
$\{\,p(\w) \mid \w \in \words\,\}$ is regarded as a set of indeterminates.

The design matrices $\A{3}$ and $\A{4}$ from Model \eqref{model3} and Model \eqref{model4}, also define two toric ideals, which can be defined in a similar way as those from models \eqref{model1} and \eqref{model2} respectively, or they can also be regarded as a specialization of the ideals $I_{\A{3}}$ and $I_{\A{4}}$ respectively, when we set $\beta_{i,i}=0$ for all $i\in[S]$.

Let $\cA \in \Z^{d \times m}$ be a design matrix for one of our models.
Then the set of all contingency tables (data vectors) satisfying a given marginals $\ve b \in
\Z^{d}_+$ is called a \emph{fiber}  which we denote by $\F_{\ve b} = \{\,\ve x
\in \Z^{m}_+ \mid \cA (\ve x) = \ve b \,\}$.  A \emph{move} $\ve z \in \Z^{m}$ is an
integer vector satisfying $\cA (\ve z) = 0$.  A {\em Markov basis} for a model
defined by the design matrix $\cA$ is defined as a finite set of moves $\mathcal
Z$ satisfying that for all $\ve b$ and all pairs $\ve x, \ve y \in \F_{\ve b}$
there exists a sequence $\ve z_1,\ldots,\ve z_K \in \mathcal Z$ such that

\begin{equation*}
\ve y = \ve x + \sum_{k=1}^K \ve z_k, \quad \ve x + \sum_{k=1}^l \ve
z_k \geq \ve 0, \; l=1,\ldots,K.
\end{equation*}
A
{\em minimal Markov basis} is a Markov basis which is minimal in terms of inclusion.
 See Diaconis and Sturmfels\cite{diaconis-sturmfels} for
more details on Markov bases and their toric ideals.  

\subsection{State Graph}

Given any multiset
$W \in \MS(\words)$ we consider the directed multigraph called the \emph{state graph} 
$G(W)$. The vertices of $G(W)$ are given by the states $[S]$ and the directed edges $i\to j$ are given
by the transitions from state $i$ to $j$ in $\ve w \in W$. Thus, we regard $\ve w \in W$ as a path of length $T-1$ in $G(W)$. See Figure \ref{stateexone} for an 
example of the state graph $G(W)$ of the multiset of paths with length $3$.
$W = \{(112),(223),(331)\}$ where $[S]=\{1, 2, 3\}$ and $T = 3$. However, notice
that the state graph in Figure \ref{stateexone} is the same for
another multiset of paths $\overline{W} = \{(122),(233),(311)\}$.

\begin{figure}[!htp]
\begin{center}
\scalebox{0.6}{
\includegraphics{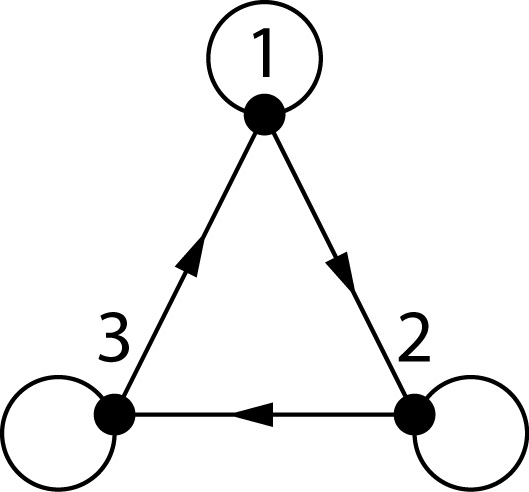}
}
\end{center}
\caption{The state graph $G(W)$ of $W = \{(112),(223),(331)\}$. Also the state graph $G(\overline{W})$ of $\overline{W} = \{(122),(233),(311)\}$.}
\label{stateexone}
\end{figure}

\begin{proposition} $\phantom{A}$
\label{fiberstateequiv}
\begin{enumerate}
\item Let $W,\overline W \in \MS(\words)$. $\A{2} (W)  = \A{2} (\overline W)$ if and only if $G(W) = G(\overline W)$. 
\item Let $W,\overline W \in \MS(\wordsnl)$. $\A{4} (W)  = \A{4} (\overline W)$ if and only if $G(W) = G(\overline W)$. 
\end{enumerate}
\end{proposition}

The equivalence of the state graphs is not sufficient to show that $W$ and $W'$
are in the same fiber when the initial states (Model \eqref{model1} and Model
\eqref{model3}) are considered. We extend the definition of the state graph to
incorporate the initial states as follows: Let $W \in \MS(\words)$ and define
the \emph{marked state graph} $\overline G(W)$ to be the same as $G(W)$ but with the additional condition that every vertex $v$ of $\overline G(W)$ is marked with the number of words $ w\in W$ that start at state $v$. We illustrate this definition in Figure
\ref{markedstateexone}.

\begin{figure}[!htp]
\begin{center}
\scalebox{0.6}{
\includegraphics{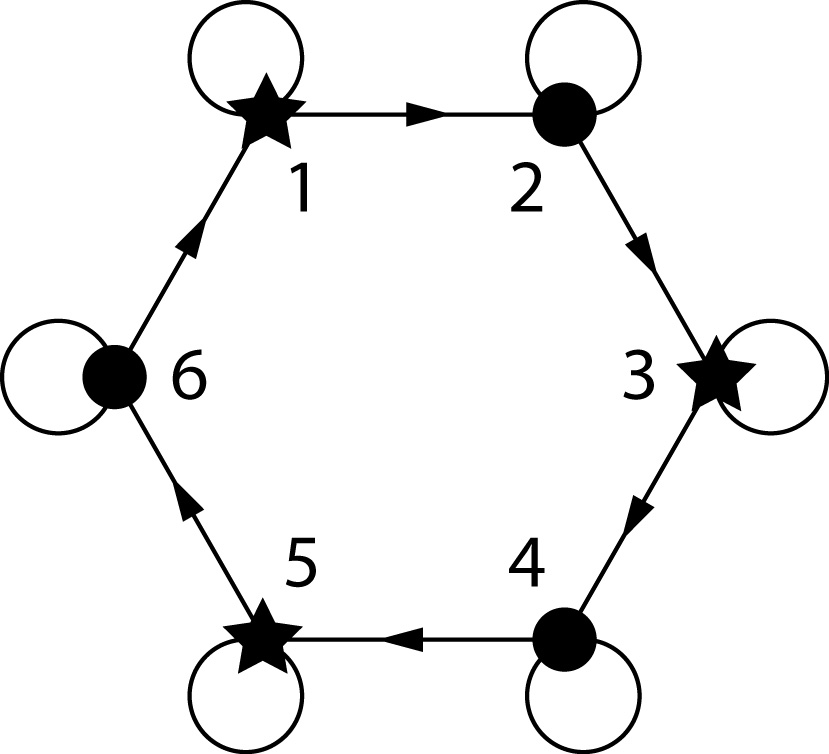}
}
\end{center}
\caption{The marked state graph $\overline G(W)$ of
    $W = \{(11223),(33445),(55661)\} \in \cp(\words)$. The vertices
    $\{1,3,5\}$ are marked once since they are the initial states of the paths
    in $W$.}
\label{markedstateexone}
\end{figure}

\begin{proposition} $\phantom{A}$
\label{fiberstateequivinit}
\begin{enumerate}
\item Let $W,\overline W \in \MS(\words)$. $\A{1} (W)  = \A{1} (\overline W)$ if and only if $\overline G(W) = \overline G(\overline W)$. 
\item Let $W,\overline W \in \MS(\wordsnl)$. $\A{3} (W)  = \A{3} (\overline W)$ if and only if $\overline G(W) = \overline G(\overline W)$. 
\end{enumerate}
\end{proposition}

If $i$ is a vertex of the (possibly marked) multigraph $G$ (or $\overline G$), then we let $G_{i+}$ denote
\emph{out-degree} of $i$, which is the number of directed edges leaving vertex
$i$. Similarly we let $G_{+i}$ denote the \emph{in-degree} of $i$, which is the
number of directed edges entering vertex $i$.

\section{Smith Normal Form}\label{sec:SNF}

For an integer matrix $A\in \Z^{d\times m}$, we consider the Smith normal form $D$ of $A$, which is a diagonal matrix $D$ for which there exists unimodular matrices $U\in\Z^{d\times d}$ and $V\in \Z^{m\times m}$, such that $UAV=D$. The Smith normal form encodes the $\Z$-module structure of the abelian group $\Z A :=\{n_1 A_1+\cdots+n_m A_m \mid n_i\in\Z\}$. Some additional material about the Smith normal form for matrices with entries over a PID can be found in the book of C. Yap\cite{yap2000}.  In this section, we explore the Smith normal form when the matrices to consider are those design matrices of the THMC model from the previous sections. This study will be important for Sections \ref{sec:normality} and \ref{sec:polytope}, where we study normality of the toric ideal associated to the model.

Our first result is the characterization of the Smith normal form for the design matrix of the model \eqref{model2}.

\begin{proposition}\label{SNFmodel2}
For any $S$ and $T$, the Smith normal form of the design matrix $\A{2}_{S,T}$ is $D=\diag(1,\ldots,1,T-1)$.
\end{proposition}

\begin{proof}
For $S$ and $T$ fixed, let $\A{2}$ be the design matrix. We order the columns of $\A{2}$ so that the first $S^2$ columns correspond to paths of the form $11\ldots 1s_is_j$ for $s_i,s_j\in[S]$ ordered lexicographically. To simplify the notation in this proof, when $w = 11\ldots1s_is_j$, we write only $w=(s_is_j)$, as the prefix $11\ldots1$ is understood. For example, when $S = 3$, the first $9$ columns of $\A{2}$ are:
\[
\begin{array}{c|ccccccccc}
 & (11) & (12) & (13) & (21) & (22) & (23) & (31) & (32) &
(33)\\\hline
11 & T - 1 & T-2 & T-2 & T-3 & T-3 & T-3  &  T-3 & T-3 & T-3 \\
12 & 0 & 1 & 0 & 1 & 1 & 1 & 0 & 0 & 0 \\
13 & 0 & 0 & 1 & 0 & 0 & 0 & 1 & 1 & 1\\
21 & 0 & 0 & 0 & 1 & 0 & 0 & 0 & 0 & 0 \\
22 & 0 & 0 & 0 & 0 & 1 & 0 & 0 & 0 & 0\\
23 & 0 & 0 & 0 & 0 & 0 & 1 & 0 & 0 & 0 \\
31 & 0 & 0 & 0 & 0 & 0 & 0 & 1 & 0 & 0\\
32 & 0 & 0 & 0 & 0 & 0 & 0 & 0 & 1 & 0\\
33 & 0 & 0 & 0 & 0 & 0 & 0 & 0 & 0 & 1\\
\end{array}
\]

We will use this example throughout this proof, in order to illustrate the arguments of the proof. 

Now, we apply the row operation that add all other rows to the first row. This row operation is encoded by a unimodular matrix $U$ that has ones in the first row and in the diagonal, and zero in all other entries, 
\[
U = \left(
\begin{array}{cccccc}
1 & 1 & 1 &\cdots & 1 & 1\\
0 & 1 & 0 & \cdots & 0 & 0\\
0 & 0 & 1 & \cdots &  0 & 0\\
0 & 0 & 0 & \ddots &  1 & 0\\
0 & 0 & 0 & \cdots &  0 & 1\\
\end{array}
\right).
\]

For the case $S=3$ from above, the first 9 columns of $U\A{2}$ are:  
\[
\begin{array}{c|ccccccccc}
 & (11) & (12) & (13) & (21) & (22) & (23) & (31) & (32) &
(33)\\ \hline
11 & T - 1& T-1& T-1 & T-1&T-1&T-1 & T-1&T-1&T-1\\
12 & 0 & 1 & 0 & 1 & 1 & 1 & 0 & 0 & 0 \\
13 & 0 & 0 & 1 & 0 & 0 & 0 & 1 & 1 & 1\\
21 & 0 & 0 & 0 & 1 & 0 & 0 & 0 & 0 & 0 \\
22 & 0 & 0 & 0 & 0 & 1 & 0 & 0 & 0 & 0\\
23 & 0 & 0 & 0 & 0 & 0 & 1 & 0 & 0 & 0 \\
31 & 0 & 0 & 0 & 0 & 0 & 0 & 1 & 0 & 0\\
32 & 0 & 0 & 0 & 0 & 0 & 0 & 0 & 1 & 0\\
33 & 0 & 0 & 0 & 0 & 0 & 0 & 0 & 0 & 1\\
\end{array} 
\]

Notice that the entries in the first row are all equal to $T-1$, since the column sum of the design matrix is precisely $T-1$.
Now, for every $s\neq 1$, we subtract the column $(1,s)$ from the columns $(s,s_j)$ for all $s_j\in [S]$. For instance, in our example, we subtract the column $(12)$ from the columns
$(21)$, $(22)$, and $(23)$, and similarly the column $(13)$ from the columns $(31),\, (32)$, and $(33)$, to get:
\[
\begin{array}{c|ccccccccc}
 & (11) & (12) & (13) & (21) & (22) & (23) & (31) & (32) &
(33)\\\hline
11 & T - 1& T-1& T-1 & 0& 0& 0 & 0& 0& 0\\
12 & 0 & 1 & 0 & 0 & 0 & 0 & 0 & 0 & 0 \\
13 & 0 & 0 & 1 & 0 & 0 & 0 & 0 & 0 & 0\\
21 & 0 & 0 & 0 & 1 & 0 & 0 & 0 & 0 & 0 \\
22 & 0 & 0 & 0 & 0 & 1 & 0 & 0 & 0 & 0\\
23 & 0 & 0 & 0 & 0 & 0 & 1 & 0 & 0 & 0 \\
31 & 0 & 0 & 0 & 0 & 0 & 0 & 1 & 0 & 0\\
32 & 0 & 0 & 0 & 0 & 0 & 0 & 0 & 1 & 0\\
33 & 0 & 0 & 0 & 0 & 0 & 0 & 0 & 0 & 1\\
\end{array}
\]

Then, we subtract the first column $(11)$ to those columns indicated by $(1,s)$ for $s\neq 1$.
In our example, we subtract the first column to the columns $(12)$ and $(13)$ to have:
\[
\begin{array}{c|ccccccccc}
 & (11) & (12) & (13) & (21) & (22) & (23) & (31) & (32) &
(33)\\ \hline
11 & T - 1&  0& 0 & 0& 0& 0 & 0& 0&0\\
12 & 0 & 1 & 0 & 0 & 0 & 0 & 0 & 0 & 0 \\
13 & 0 & 0 & 1 & 0 & 0 & 0 & 0 & 0 & 0\\
21 & 0 & 0 & 0 & 1 & 0 & 0 & 0 & 0 & 0 \\
22 & 0 & 0 & 0 & 0 & 1 & 0 & 0 & 0 & 0\\
23 & 0 & 0 & 0 & 0 & 0 & 1 & 0 & 0 & 0 \\
31 & 0 & 0 & 0 & 0 & 0 & 0 & 1 & 0 & 0\\
32 & 0 & 0 & 0 & 0 & 0 & 0 & 0 & 1 & 0\\
33 & 0 & 0 & 0 & 0 & 0 & 0 & 0 & 0 & 1\\
\end{array}
\]

Lastly, we subtract the first column $(11)$ from all the other columns. Since these last operations involved only columns, we can encode them with a unimodular matrix $V'$.
In this way, we can write
\[
U \A{2} V' = \left(
\begin{array}{ccc}
T-1 &  0 & 0\\
0 & I_{S^2 - 1} & \bar{A}\\
\end{array}
\right) 
\]
where $I_{S^2 - 1} $ is the $(S^2-1) \times (S^2-1) $
identity matrix and $\bar{A}$ is some $(S^2-1)\times (S^{T-1}-S^2)$ integer matrix.

We can use the first $S^2$ columns, to bring $U\A{2}V'$ to the form $D=\diag(T-1,1,\ldots,1)$, so there exists a unimodular matrix $V$ such that
\[
U\A{2}V = \left(
\begin{array}{ccc}
T-1 &  0 & 0\\
0 & I_{S^2 - 1} & 0\\
\end{array}
\right) =D .
\]
\end{proof}

Although the proof we just presented is, by far, not the shortest argument we could find, we consider this constructive proof was worth to present, as it shows the special structure of the unimodular matrix $U$. We present now an important result related to the integer lattice generated by the columns of the design matrix.

Let $\Z \A{2}$ be the integer lattice generated by the columns of $\A{2}$, i.e.
\[
\Z \A{2} = \left\{\sum_{\w\in\words} \alpha_\w \A{1}_\w \ |\ \alpha_\w\in\Z\right\},
\]
where $\A{2}_\w$ is the column of $\A{2}$ corresponding to the path $\w$.

\begin{lem}\label{colspanModel2}
Let $\A{2}$ be the design matrix for some $S$ and $T$; then, $\y = (y_1,\ldots,y_{S^2})^\top\in \Z\A{2}$ if and only if $y_1+\cdots+y_{S^2}\equiv 0 \mod (T-1)$.
\end{lem}

\begin{proof}
Let $D$ be the Smith normal form of $\A{2}$, and $U,V$ unimodular matrices such that $U\A{2} V=D$, as in the proof of Proposition \ref{SNFmodel2}.  For $\ve y \in \Z \A{2}$, we write $\ve y = \A{2} \ve z$ for some $\ve z \in
\Z^{S^T}$.  Then $\ve y = U^{-1} D V^{-1}\ve z$.  Since $V^{-1}$ is
unimodular 
\[
\{V^{-1}\ve z| \ve z \in \Z^{S^T} \} = \Z^{S^T}.
\]
Write $\bar{\ve z} := V^{-1}\ve z = (\bar{z}_1, \ldots ,
\bar{z}_{S^T})^\top$.  Then,
\[
U \ve y = D \bar{\ve z} = \left( \begin{array}{c}
(T-1) \bar{z}_1\\
\bar{z}_2\\
\vdots\\
\bar{z}_{S^2}\\
\end{array}\right) \, .
\]

Furthermore, we have the following integer system of equations
\begin{equation}\label{eq_1}
\begin{array}{rcl}
y_1 + y_2 + \cdots + y_{S^2} & = &(T-1) \bar{z}_1\\ 
y_2 & = & \bar{z}_2\\
& \vdots & \\
y_{S^2} & = & \bar{z}_{S^T}.\\
\end{array}
\end{equation}
The result follows from this system of equations.
\end{proof}

We now prove an analogous result for Model $\A{4}$. For this, we will distinguish some pairs of paths $(P,Q)$ whose transition count is the same except for one, as explained in the Remark \ref{rem:sufficient}

\begin{defn}\label{def:pivotpaths}
Let $S \geq 3$, $T \geq 4$, and $i,j,k \in [S]$ be three states which are pairwise distinct. Define the pairs of paths

\begin{equation}\label{path1}
    \left(P^{ji}_{ki},Q^{ji}_{ki}\right)
    := \left\{ \begin{array}{rl}
        \big(\{ (ijij\cdots jik)\},\{(ikijij\cdots ijij)\}\big) & \text{ if  } T \text{ even}\\
        \big(\{ (ikjij\cdots jik)\},\{(ikikjij\cdots jiji)\}\big) & \text{ if  } T \text{ odd}\\
    \end{array}\right.
\end{equation}

\begin{equation}\label{path2}
    \left(P^{ki}_{kj},Q^{ki}_{kj}\right)
    := \left\{ \begin{array}{rl}
        \big(\{ (kijij\cdots ijij)\},\{(kjiji \cdots ijij)\}\big) & \text{ if  } T \text{ even}\\
        \big(\{ (kikjij\cdots jiji)\},\{(kjikji \cdots ijij)\}\big) & \text{ if  } T \text{ odd}\\
    \end{array}\right.
\end{equation}
\end{defn}


\begin{ex}
Conforming with Definition \ref{def:pivotpaths}, for $S\geq 3$ and $i=2,\, j=1,\, k=3$; we have:
\begin{itemize}
\item When $T=6$; then, $P^{12}_{32} = 212123$, $Q^{12}_{32}=232121$ according to \eqref{path1} and $P^{32}_{31} = 321212,\, Q^{32}_{31} = 312121$ according to \eqref{path2}.
\item For $T=7$; then, $P^{12}_{32} = 2312123$, $Q^{12}_{32} = 2323121$ according to \eqref{path1} and $P^{32}_{31} = 3231212,\, Q^{32}_{31} = 3123121$ according to \eqref{path2}.
\end{itemize}
\end{ex}

\begin{rem}\label{rem:sufficient}
Let $\mathcal A$ be any of the design matrices for Models \eqref{model3} and \eqref{model4}. Recall that if $ w$ is a path, then $\mathcal A(w)$ denotes the column of $\mathcal A$ indicated by $w$. We then observe the following:
\begin{enumerate}
\item If $P^{ji}_{ki}$ and $Q^{ji}_{ki}$ are as in \eqref{path1}, and we let $v = \mathcal A (P^{ji}_{ki}) - \mathcal A (Q^{ji}_{ki})$; then,  $v$ satisfies $v_{ji} = 1$, $v_{ki} = -1$ and $v_{st} = 0$ for all other $s,t \in [S]$.  
\item If $P^{ki}_{kj}$ and $Q^{ki}_{kj}$ are as in \eqref{path2}, and we let $ v = \mathcal A (P^{ki}_{kj}) - \mathcal A (Q^{ki}_{kj})$; then, $ v$ satisfies $v_{ki} = 1$, $v_{kj} = -1$ and $v_{st} = 0$ for all other $s,t \in [S]$.  
\end{enumerate}
This is how our notation indicates, for instance, where are the two nonzero entries of the difference $\cA(P^{ji}_{ki}) - \cA(Q^{ji}_{ki})$; even more, it indicates that the $(ji)$th coordinate of this difference is 1 and the $(ki)$th coordinate is -1.
%
\end{rem}

\begin{proposition}\label{SNFmodel4}
For $S \geq 3$ and $T \geq 4$, the Smith normal form of the design matrix $\A{4}_{S,T}$ is $D=\diag(1,\ldots,1,T-1)$.
\end{proposition}

\begin{proof}
First we show that for all $i,j \in [S]$ such that $\{i,j\} \neq \{1,2\}$, there exists a path $w \in [S]^T_{NL}$ for which the column $\A{4}(w)$, using only column operations, can be brought to the form $v$ where $v\in \Z^{S(S-1)}$, with $v_{12}=1$, $v_{ij} = -1$, and $v_{st} = 0$ for all other $s,t \in [S]$; we call the path $w$ a \emph{pivot path}. We will study four cases for the vector $v$, depending on the states $i,j\in[S]$ for which $v_{ij}=-1$.

\begin{enumerate}
\item Case $i=1, \, j\geq 2$: Let the pivot path $w$ be $P^{12}_{1j}$ from \eqref{path2} of Definition \ref{def:pivotpaths}. Then, as indicated in Remark \ref{rem:sufficient}, $v=\A{4}(P^{12}_{1j}) - \A{4}(Q^{12}_{1j})$ satisfies $v_{12}=1$, $v_{1j}=-1$, and $v_{st}=0$ for all other $s,t\in [S]$.\\

\item Case $i\geq 2, \, j=1$: Let $k\in [S]$ satisfying $k\neq 1$ and $k\neq i$. In this case, let the pivot path be $w=P^{ki}_{k1}$ from \eqref{path2} of Definition \ref{def:pivotpaths}. Then, 
$v=\A{4}(P^{ki}_{k1}) + \A{4}(P^{1k}_{ki}) + \A{4}(P^{12}_{1k}) - \A{4}(Q^{ki}_{k1}) - \A{4}(Q^{1k}_{ki}) - \A{4}(Q^{12}_{1k})$ satisfies $v_{12}=1$, $v_{i1}=-1$, and $v_{st}=0$ otherwise.\\

\item Case $i\geq 3,\, j=2$: Let the pivot path be $w=P^{12}_{i2}$ from \eqref{path1} of Definition \ref{def:pivotpaths}. Then, $v=\A{4}(P^{12}_{i2})-\A{4}(Q^{12}_{i2})$ satisfies $v_{12}=1$, $v_{i2}=-1$, and $v_{st}=0$ for all other $s,t\in [S]$.\\

\item Case $i\geq 2,\, j\geq 3$: In this case we cover all other cases not covered before. For this, we let the pivot path be $w=P^{1j}_{ij}$ from \eqref{path1} of Definition \ref{def:pivotpaths}. Then, $v=\A{4}(P^{1j}_{ij})+\A{4}(P^{12}_{1j})-\A{4}(Q^{1j}_{ij})-\A{4}(Q^{12}_{1j})$ satisfies $v_{12}=1$, $v_{ij}=-1$, and $v_{st}=0$ otherwise.
\end{enumerate}

There are $S(S-1)-1$ pivot paths listed above; ordering these columns to be first in $\A{4}$ and using the column operations from above, the design matrix $\A{4}$ can be brought into the form
\begin{equation*}
\left(
\begin{array}{rrrrrrrrrr}
 1  & 1  & 1  & \cdots & 1  & 1 \\
 -1 & 0  & 0  & \cdots & 0  & 0 \\
 0  & -1 & 0  & \cdots & 0  & 0 \\
 0  & 0  & -1 & \cdots & 0  & 0 & A_l & A_{l+1} & \cdots & A_{m}\\
\vdots & \vdots & \vdots & \ddots & \vdots & \vdots\\
 0  & 0  & 0  & \cdots & -1  & 0 \\
 0  & 0  & 0  & \cdots & 0  & -1\\
\end{array}
\right)
\end{equation*}
where $A_l,\cdots,A_{m}$ are unmodified columns of $\A{4}$.  The column $A_l$ has non-negative integral entries that sum to $T-1$. By adding integer multiples of the first $\pm 1$ vectors in the left of the modified $\A{4}$ matrix, $A_l$ can be transformed to $(T-1,0,\ldots,0)^\top$. That is, $\A{4}$ can be transformed by only column operations into
\begin{equation*}
\left(
\begin{array}{rrrrrrcrrr}
 1  & 1  & 1  & \cdots & 1  & 1 & T-1\\
 -1 & 0  & 0  & \cdots & 0  & 0 & 0\\
 0  & -1 & 0  & \cdots & 0  & 0 & 0\\
 0  & 0  & -1 & \cdots & 0  & 0 & 0 & A_{l+1} & \cdots & A_{m}\\
\vdots & \vdots & \vdots & \ddots & \vdots & \vdots\\
 0  & 0  & 0  & \cdots & -1  & 0 & 0\\
 0  & 0  & 0  & \cdots & 0  & -1 & 0\\
\end{array}
\right).
\end{equation*}
By adding all the rows, besides the first, of the above matrix to the first row, and subtracting the column $(T-1,0,\ldots,0)^\top$ from all the columns to its right we get
\begin{equation*}
\left(
\begin{array}{rrrrrrcccc}
 0  & 0  & 0  & \cdots & 0  & 0 & T-1 & 0 & \cdots & 0\\
 -1 & 0  & 0  & \cdots & 0  & 0 & 0\\
 0  & -1 & 0  & \cdots & 0  & 0 & 0\\
 0  & 0  & -1 & \cdots & 0  & 0 & 0 & \bar A_{l+1} & \cdots & \bar A_{m}\\
\vdots & \vdots & \vdots & \ddots & \vdots & \vdots\\
 0  & 0  & 0  & \cdots & -1  & 0 & 0\\
 0  & 0  & 0  & \cdots & 0  & -1 & 0\\
\end{array}
\right).
\end{equation*}
These operations can be encoded by unimodular matrices $U$ and $V$:
\[
U \A{4} V' = \left(
\begin{array}{ccc}
T-1 &  0 & 0\\
0 & I_{S(S-1) - 1} & \bar{A}\\
\end{array}
\right) 
\]
where $U$ encode the row operation described above, so it is of the form:
\[
U = \left(
\begin{array}{cccccc}
1 & 1 & 1 &\cdots & 1 & 1\\
0 & 1 & 0 & \cdots & 0 & 0\\
0 & 0 & 0 & \ddots &  0 & 0\\
0 & 0 & 0 & \cdots &  1 & 0\\
0 & 0 & 0 & \cdots &  0 & 1\\
\end{array}
\right) 
\]
The rest of the proof is the same as the one of Proposition~\ref{SNFmodel2}.
\end{proof}


\begin{lem} \label{lem:latticeB}
$\ve y = (y_1, \ldots , y_{S^2})^\top \in \Z \A{4}$ if and only if $y_1 +
\cdots + y_{S^2} \equiv 0 \mod (T-1)$.
\end{lem}

\begin{proof}
Since in our construction we get the same unimodular matrix $U$ as in Proposition~\ref{SNFmodel2}, then the same proof of Lemma~\ref{colspanModel2} applies.
\end{proof}


\section{Semigroup}\label{sec:normality}

As studied in the last section, to an integer matrix $A\in\Z^{d\times m}$ we associate an integer lattice $\Z A = \{n_1 A_1 +\cdots + n_m A_m \mid n_i\in \Z\}$. We can also associate the semigroup $\N A := \{n_1 A_1 + \cdots + n_m A_m \mid n_i \in \N\}$. We say that the semigroup $\N A$ is normal when $\ve x \in \N A$ if and only if there exist $\ve y\in \Z^d$ and $\alpha\in \R_{\geq0}^d$ such that $\ve x = A \ve y$ and $\ve x = A \alpha$. See Miller and Sturmfels\cite{MS2005} for more details on normality.

In this section we will discuss the normality/non-normality of the
semigroup generated by the columns of the design matrix for each
model; $\A{1}$, $\A{2}$, $\A{3}$,  $\A{4}$.  

\subsection{Model \eqref{model1}}

\begin{lem}\label{non-normal1}
The semigroup generated by the columns of the
design matrix $\A{1}$ is not normal for $S \geq 3$ and $T \geq 4$.
\end{lem}
\begin{proof}
First, for $S = 3$ and $T = 4$, we look at the semigroup $\N \A{1}$ generated by the columns of $\A{1}$. We ordered the indices of the coordinates lexicographically, i.e., as $1,\, 2,\,
3,\, (11),\, (12),\, (13), \, (21),\,  (22),\, (23),\,$ $ (31),\, (32),\, (33)$.
We claim that the vector $\ve h:= (1, 0, 1, 1, 1, 0, 0, 0,
1, 0, 2, 1)^\top$ is indeed not in the semigroup $\N\A{1}$, as one can verify that there does not exist an integral
non-negative solution for the system
\[
\A{1} \ve x = \ve h, \, \ve x \geq 0.
\].    However, we have
\[
\ve h = \frac{1}{2} \A{1}_{(1112)} + \frac{1}{2} \A{1}_{(1232)} +
\frac{1}{2} \A{1}_{(3232)} + \frac{1}{2} \A{1}_{(3332)}. 
\]
Thus $\ve h$ is in the saturation of $\N\A{1}$.
Also since 
\[
\ve h =  \A{1}_{(1123)} +  \A{1}_{(3332)} +  \A{1}_{(3232)} - \A{1}_{(3323)},
\]
we know that $\ve h$ is in the lattice generated by the columns of the
matrix $\A{1}$.  Thus $\ve
h$ is in the difference between $\N\A{1}$ and its saturation.

Based on this case, we show now that for $T>4$, there is a vector in the saturation of $\N \A{1}$ but not in
$\N \A{1}$.   Let $\ve h := (1, 0, 1, T-3, 1, 0, 0, 0, 1, 0, 2, T-3)^\top$ and notice that 
\[
\ve h = \frac{1}{2} \A{1}_{(1 \ldots 1112)} + \frac{1}{2} \A{1}_{(1\ldots
  1232)} + \frac{1}{2} \A{1}_{(3\ldots 3232)} +\frac{1}{2} \A{1}_{(3 \ldots 3332)},
\]
where $(1 \ldots 1112)$ is the path with $T - 4$ many 1's in front of
the path $(1112)$, $(1\ldots 1232)$ is  the path with $T - 4$ many 1's in front of
the path $(1232)$, $(3\ldots 3232)$ is the path with $T - 4$ many 3's
in front of the path $(3232)$, and $(3 \ldots 3332)$ is the path with $T - 4$ many 3's
in front of the path $(3332)$.
Thus $\ve h$ is in the saturation of $\N\A{1}$.  Also 
\[
\begin{array}{lll}
\ve h 
& = & (1,0,1,T-3,1,0,0,0,1,0,2,T-3)^\top\\
& = & \A{1}_{(1\ldots 1123)} +  \A{1}_{(3\dots 3332)} +  \A{1}_{(3\ldots
  3232)} - \A{1}_{(3\ldots 3323)}\\
&= & (1,0,0,T-3,1, 0, 0, 0, 1, 0, 0, 0)^\top\\
&+ & (0,0,1,0,0,0,0,0,0,0,1,T-2)^\top\\
&+ & (0,0,1,0,0,0,0,0,1,0,2,T-4)^\top\\
&- & (0,0,1,0,0,0,0,0,1,0,1,T-3)^\top.
\end{array}
\] 
Thus $\ve h$ is in the integer lattice $\Z\A{1}$
but there does not exist an integral
non-negative solution for the system
\[
\A{1} \ve x = \ve h, \, \ve x \geq 0.
\]
Thus $\ve h$ is not in $\N\A{1}$.

For $S>3$, we just observe that $\Z\A{1}_{S,T}\subseteq \Z\A{1}_{S+1,T}$.
\end{proof}

When $S=2$, the semigroup $\N\A{1}$ seem to be normal, as stated in the following conjecture. \begin{conj}\label{conj:normal1}
The semigroup generated by the columns of the
design matrix $\A{1}$ is normal for $S = 2$ and $T \geq 2$.
\end{conj}

We verified this conjecture using the software {\tt normaliz}\cite{normaliz} for $T=1,\ldots,100$. 

\subsection{Model \eqref{model2}}

Different from Model \eqref{model1}, the Model \eqref{model2} does not satisfy the normality condition.
\begin{lem}
The semigroup generated by the columns of the
design matrix $\A{2}$ is not normal for $S \geq 2$ and $T \geq 3$.
\end{lem}

\begin{proof}
For $T \geq 3$ and $S = 2$, let $\ve h = (1, 0, 0,  T - 2)^\top$.  
We want to show that $\ve h$ is not in the semigroup $\N\A{2}$ but it is in the saturation of $\N\A{2}$.

Note that $\ve h$ can be written as
\[
\frac{1}{T- 1}  (T - 1, 0, 0, 0)^\top + \frac{T - 2}{T - 1} (0, 0, 0, T - 1)^\top
= \frac{1}{T - 1} \A{2}_{(111\ldots 1)} + \frac{T - 2}{T - 1} \A{2}_{(222\ldots 2)},
\]
where $(111\ldots 1)$ is the path of $T$ many 1's, and $(222\ldots 2)$ is
the path of $T$ many 2's. 
We can also write $\ve h$  as 
\begin{eqnarray*}
(T - 1, 0, 0, 0)^\top - (T - 2, 1, 0, 0)^\top + (0, 1, 0, T - 2)^\top \\
=\A{2}_{(1, 1, \ldots , 1, 1)} - \A{2}_{(1, 1, \ldots, 1, 2)} + \A{2}_{(1, 2, \ldots, 2, 2)}
\end{eqnarray*}
where $(1, 1, \ldots , 1, 1)$ is the path of $T - 1$ many 1's, $(1, 1,
\ldots , 1, 2)$ is the path of $T - 2$ many 1's and one 2, and $(1, 2,
\ldots , 2, 2)$ is the path of $T - 2$ many 2's and one 1.  Thus $\ve
h$ is in the lattice generate by the columns of $\A{2}$.
Thus $\ve h$ is in the saturation of $\N\A{2}$ and in the integer lattice $\Z\A{2}$, but there does not exist an integral
non-negative solution for the system
\[
\A{2} \ve x = \ve h, \, \ve x \geq 0
\]
because $\ve h$ consist of only one transition of the form 11, and $T - 2$ transitions of the form 22. Thus $\ve h$ is not in $\N\A{2}$ but $\ve h$ is a lattice point in the
cone generated by the columns of $\A{2}$. 

For $S>2$, we just set all the transitions involving the state $s>2, s \in
[S]$ to be zero.
\end{proof}

\subsection{Model \eqref{model3}}

For $S = 3$ and for $T = 4, \ldots , 9$, we have computed the Hilbert 
basis for the cone generated by the columns of the design matrix $\A{3}$ over
the lattice generated by the columns of the matrix $\A{3}$ using {\tt normaliz}. The running time of {\tt normaliz} was under two seconds for
all data sets. The most time consuming part in our experiment was generating the design matrices.
It turns out that the set of columns of $\A{3}$ contains the Hilbert basis for
all cases, which implies normality. See Table \ref{tab:hb_fvec_model3} in Section~\ref{sec:computations} for more details. Thus we
have the following conjecture.

\begin{conj}\label{loopless_noinitial_normal1}
For $S = 3$ and for $T \geq 4$, the semigroup generated by the columns
of the design  matrix $\A{3}$ is
normal. 
\end{conj}

\subsection{Model \eqref{model4}}

Similarly, using {\tt normaliz}, we have computed the Hilbert basis for cone generated by $\A{4}$ over the lattice $\Z\A{4}$ for $T = 4, \ldots , 15$. The running time of {\tt normaliz} was again under two seconds for all data sets. 
It turns out that the set of columns of $\A{4}$ contains already the Hilbert basis for
all cases, which implies normality. In Table \ref{tab:hb_fvec_model4} of Section \ref{sec:computations} we present these results, which support the following conjecture. 

\begin{conj}\label{loopless_noinitial_normal2}
For $S = 3$ and for $T \geq 4$, the semigroup generated by the columns
of the design  matrix $\A{4}$ is
normal. 
\end{conj}

\section{Polytope Structure}\label{sec:polytope}

We recall some necessary definitions from polyhedral geometry and we refer the reader to the book of Schrijver~\cite{Schrijver1986Theory-of-linea} for more details. The \emph{convex hull} of $\{\ve
a_1,\ldots,\ve a_m\} \subset \R^n$ is defined as 
\begin{equation*}   
\conv (\ve a_1,\ldots,\ve a_m) := \left\{\, \ve x \in \R^n \mid \ve x = \sum_{i=1}^m \lambda_i \ve a_i, \, \, \sum_{i=1}^m \lambda_i = 1, \, \lambda_i \geq 0 \, \right\}.  
\end{equation*} 

A \emph{polytope} $\cP$ is the convex hull
of finitely many points. We say $F \subseteq \cP$ is a {\em face} of the polytope
$\cP$ if there exists a vector $\ve c$ such that $F = \arg \max_{\ve x \in \cP}
\ve c \cdot \ve x$. Every face $F$ of $\cP$ is also a polytope. If the
dimension of $\cP$ is $d$, a face $F$ is a \emph{facet} if it is of dimension
$d-1$. For $k \in \N$, we define the $k$-th dilation of $P$ as $k \cP := \left\{\, k \ve x \mid \ve x \in \cP
,\right\}$. A point $\ve x \in \cP$ is a \emph{vertex} if and only if it can
not be written as a convex combination of points from $\cP \backslash \{\ve x\}$.

The \emph{cone} of $\{\ve a_1,\ldots,\ve a_m\} \subset \R^n$ is defined as 
\begin{equation*}   
\cone (\ve a_1,\ldots,\ve a_m) := \left\{\, \ve x \in \R^n \mid \ve x = \sum_{i=1}^m \lambda_i \ve a_i, \, \lambda_i \geq 0 \, \right\}.  
\end{equation*} 


We are interested in the polytopes given by the convex hull of the columns of
the design matrices of our four models. If $\cA$ is the design matrix for one
of our four models, we will simply use $\cA$ to refer to the set of columns of
the design matrix when the context is clear. Let $\Poly{1} = \conv\left( \A{1}
\right)$, $\Poly{2} = \conv\left( \A{2} \right)$, $\Poly{3} = \conv\left( \A{3}
\right)$, and $\Poly{4} = \conv\left( \A{4} \right)$. Also, we let $\Cone{1} =
\cone\left( \A{1} \right)$, $\Cone{2} = \cone\left( \A{2} \right)$, $\Cone{3} =
\cone\left( \A{3} \right)$, and $\Cone{4} = \cone\left( \A{4} \right)$. 

In this section we will focus mainly on Model \eqref{model4}.  If $\ve x \in
\R^{S(S-1)}$, we index $\ve x$ by $\{\, (i,j) \mid 1 \leq i , j \leq S,\, i \neq j\,
\}$. We define $e_{ij} \in \R^{S(S-1)}$ to be the vector of all zeros, except $1$ at index $ij$.
We also adopt the notation $x_{i+} := \sum_{j} x_{ij}$ and
$x_{+i} := \sum_{j} x_{ji}$. For any $\ve x \in \N^{S(S-1)}$ we can
define a multigraph $G(\ve x)$ on $S$ vertices, where there are $x_{ij}$ directed
edges from vertex $i$ to vertex $j$. One would like to identify the vectors $\ve x\in \N^{S(S-1)}$ for which the graph $G(\ve x)$ is a state graph. Nevertheless, observe that $x_{i+}$ is the out-degree of vertex $i$ and $x_{+i}$ is the in-degree of vertex $i$.

\begin{proposition}\label{prop:bnddeg}
If $\ve x \in \Z \A{4} \cap \Cone{4}$ then $\sum_{i \neq j} x_{ij} =
k(T-1)$ for some $k \in \N$ and $|x_{i+} - x_{+i}| \leq k$ for all
$i \in S$.
\end{proposition}

\begin{proof}
By Lemma~\ref{lem:latticeB}, we have $\sum_{i \neq j} x_{ij} = k(T-1)$ for some
$k \in \N$. Let $[A_1 \cdots A_m]$ be the columns of the design matrix $\A{4}$.
Then $\ve x \in \Cone{4}$ implies $\ve x = \sum_{i=1}^m \alpha_i A_i$ where $
 \alpha_i \geq 0$. Then 
\begin{equation*}
\sum_{i \neq j} x_{ij} = \left\| \sum_{i=1}^m \alpha_i A_i \right\|_1 = \sum_{i=1}^m \alpha_i \|A_i\|_1 = (T-1)\sum_{i=1}^m \alpha_i = k(T-1).
\end{equation*}
Thus $\sum_{i=1}^m \alpha_i = k \in \N$. Finally for $i \in S$
\begin{multline*}
| x_{i+} - x_{+i} | = \left|\alpha_1 (A_1)_{i+} + \cdots + \alpha_m (A_m)_{i+} - \alpha_1 (A_1)_{+i} - \cdots - \alpha_m (A_m)_{+i}\right| \\
= \left|\alpha_1 \big[(A_1)_{i+} - (A_1)_{+i}\big] + \cdots + \alpha_m \big[(A_m)_{i+} - (A_m)_{+i}\big] \right| \leq \sum_{i=1}^m \alpha_i = k
\end{multline*}
since $ \left| \big[ (A_l)_{i+} - (A_l)_{+i}\big] \right| \leq 1$ for all $ 1 \leq l \leq m$.
\end{proof}

Proposition \ref{prop:bnddeg} states that for $\ve x \in \Z \A{4} \cap
\Cone{4}$ the multigraph $G(\ve x)$ will have in-degree and out-degree bounded by $\|\ve x\|_1 / (T-1)$ at every vertex. This implies nice properties when $S=3$
and $\|x\|_1 = (T-1)$. Recall a path in a multigraph is \emph{Eulerian} if it
visits every edge only once.

\begin{proposition}
\label{prop:eulpath}
If $G$ is a multigraph on three vertices, with no self-loops, $T-1$ edges, and satisfying
\begin{equation*}
|G_{i+} - G_{+i}| \leq 1 \qquad \, i=1,2,3;
\end{equation*}
then, there exists an Eulerian path in $G$.
\end{proposition}

\begin{proof}
First consider the case where $G_{i+} = G_{+i}$ for $i=1,2,3$. Then, since $G$
contains no self-loops, the $T-1$ edges of $G$ consists of disjoint cycles of
the form $i \rightarrow j \rightarrow i$ or $i \rightarrow j \rightarrow k
\rightarrow i$, where $i \neq j$, $ i \neq k$ and $j \neq k$. Since $G$ only
has three vertices, every cycle has a vertex in common. Thus, there is an Eulerian
path that visits each cycle.

Suppose, without loss of generality, that $G_{1+} = G_{+1} + 1$ and $G_{2+} +
1= G_{+2}$. Then, there must exist edge $1 \rightarrow 2$ or the path $1
\rightarrow 3 \rightarrow 2$. This follows since vertex $1$ has more outgoing
edges then incoming, and they must go to either $2$ or $3$. Similarly vertex $2$
has more incoming edges than outgoing, and they must come from either $1$ or
$3$. Let $\rho$ be a path from vertex $1$ to vertex $2$ (either $1 \rightarrow 2$
or $1 \rightarrow 3 \rightarrow 2$).

Let $\widetilde G = G \backslash \rho$, that is, the graph $G$ with the edge(s)
of $\rho$ removed. Note that $\widetilde G_{+1} = \widetilde G_{1+}$,
$\widetilde G_{+2} = \widetilde G_{2+}$, $\widetilde G_{+3} =  G_{+3}$, and
$\widetilde G_{3+} = G_{3+}$. Observe that $\widetilde G_{3+} = \widetilde
G_{+3}$, as otherwise we would have a contradiction since $\widetilde G$ has
one vertex with non-zero in-degree minus out-degree. Now, as in the first case,
$\widetilde G$ consists of disjoint cycles of the form $i \rightarrow j
\rightarrow i$ or $i \rightarrow j \rightarrow k \rightarrow i$, where
$i \neq j$, $ i \neq k$ and $j \neq k$. Thus, there is an Eulerian path on
$\widetilde G$ that visits each cycle of $\widetilde G$, and we can 
append or prepend $\rho$ to get an Eulerian path on $G$.
\end{proof}

\begin{rem}
\label{rem:eulpath}
Note that every word $w \in \MS(\wordsnl)$ gives an Eulerian path in $G(\{w\})$
containing all edges.  Conversely, for every multigraph $G$ with an Eulerian path
containing all edges, there exist $ w \in \MS(\wordsnl)$ such that $G(\{w$\})$
= G$. More specifically, $w$ is the Eulerian path in $G(\{w\})$. See Figure \ref{fig:EulerEx}.
\end{rem}

\begin{figure}
\begin{center}
\scalebox{0.9}{
\includegraphics{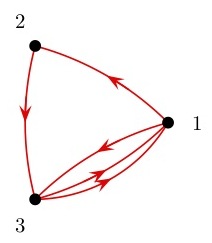}
}
\caption{For word $123131 \in \MS([3]^6_{NL})$, the state graph $G(\{123131\})$
is given above. Note that there are multiple Eulerian paths (words) in the
multigraph above.}
\label{fig:EulerEx}
\end{center}
\end{figure}

\begin{lem}
\label{lem:integinterior}
If $S=3$ and $T \geq 4$, then $\Poly{4} \cap \Z^{S(S-1)} = \A{4}$.
\end{lem}
\begin{proof}
Certainly $\Poly{4} \cap \Z^6 \supseteq \A{4}$. Let $\ve x \in \Poly{4} \cap
\Z^6$. Then $\| \ve x \|_1 \equiv 0 \mod T-1$ and $\forall i \in S$ we have
$\left| x_{i+} - x_{+i} \right| \leq 1$. Finally, considering the multigraph
$G(\ve x)$ and Proposition \ref{prop:eulpath}, we see that $\ve x$ is equal to
some column of $\A{4}$.
\end{proof}


As demonstrated in Lemma \ref{lem:integinterior}, we will find it useful to
consider $\ve x \in \N^{S(S-1)}$ as a vector, and also as a
multigraph $G(\ve x)$.

We define
\begin{equation*}
H_k := \left\{ \ve x \in \R^6 \mid \sum_{i \neq j} x_{ij} = k(T-1) \right\}.
\end{equation*}

\begin{proposition}
$\phantom{A}$\\
\begin{enumerate}
\item For $S \geq 3$, $T \geq 4$ and $k \in \N$,
\begin{equation*}
k\Poly{4} = \Cone{4} \cap H_k.
\end{equation*}
\item For $S \geq 3$ and $T \geq 4$,
\begin{equation*}
\Cone{4} \cap \Z \A{4} = \bigoplus_{k=0}^\infty \left( k\Poly{4} \cap \Z^{S(S-1)}\right).
\end{equation*}
\end{enumerate}
\end{proposition}

As we will be focusing on Model \eqref{model4} with $S=3$, we give a few
definitions specific to this case. Let $G$ be a directed multigraph on three
vertices with $T-1$ edges and no self-loops. We call a cycle $i \rightarrow j
\rightarrow i$, where $i \neq j$, a \emph{two-cycle}. Similarly we call a cycle
$i \rightarrow j \rightarrow k \rightarrow i$, where $i \neq j$, $i \neq k$,
and $j \neq k$, a \emph{three-cycle}.  We say the two-cycles $i \rightarrow j
\rightarrow i$ and $k \rightarrow l \rightarrow k$ in $G$ have
different \emph{type} if $\{ij\} \neq \{kl\}$. We let $G^2$ be the subgraph of
$G$ consisting of only the two-cycles of $G$. Similarly, we let $G^3$ be the
subgraph of $G$ consisting of only the three-cycles of $G$. By $G \backslash
G^2$, we mean the subgraph of $G$ with the edges in $G^2$ removed. Similarly
for $G \backslash G^3$. We let $|G|$ be the number of edges in $G$. We illustrate this in Figure
\ref{fig:graphdefs}.
\begin{figure}
\begin{center}
\scalebox{0.7}{
\includegraphics{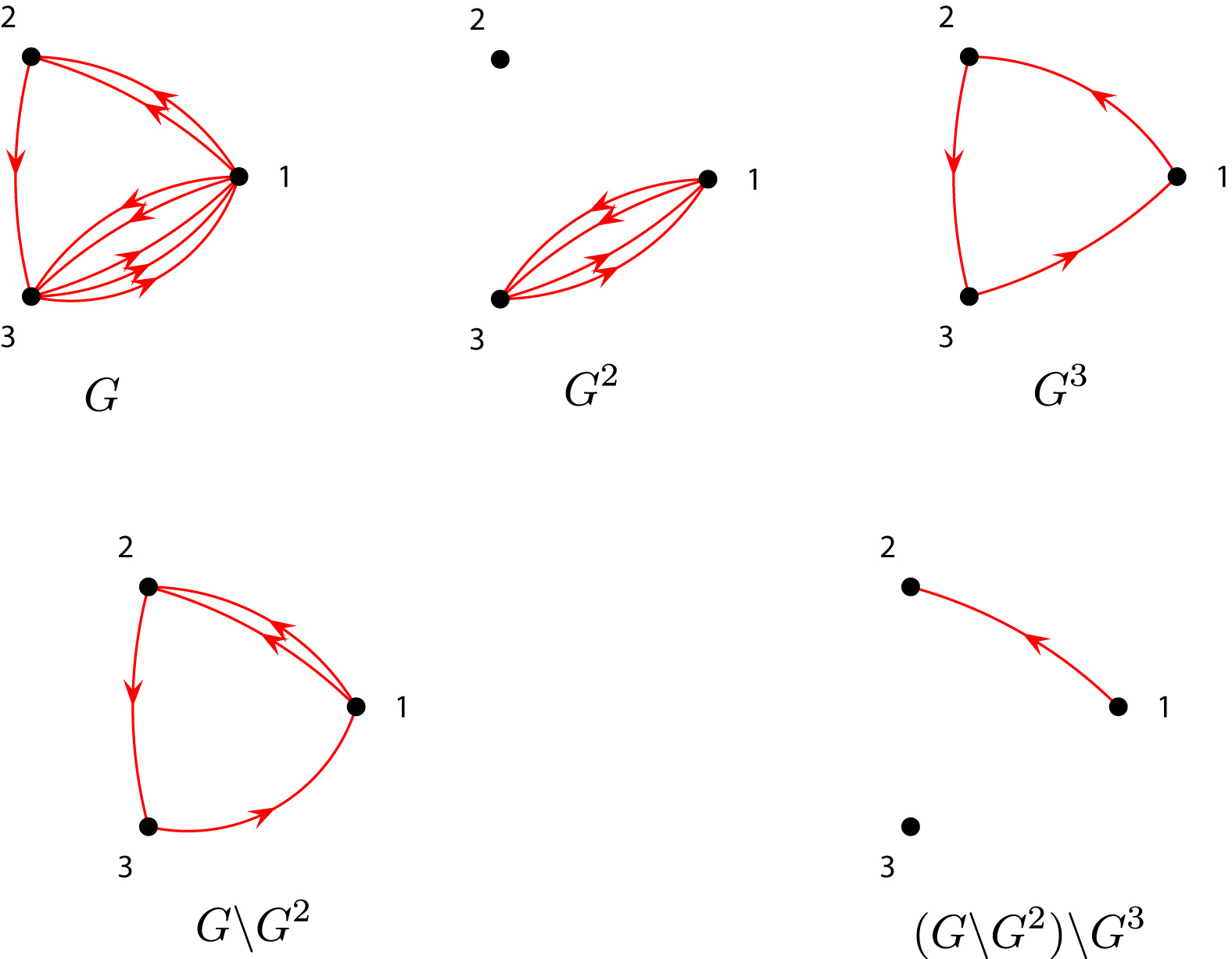}
}
\caption{A graph $G$ (from word $123131312$) and its two-cycles ($G^2$), its three-cycles ($G^3$), $G$ remove its two-cycles ($G\backslash G^2$), and $G$ remove its two-cycles then three-cycles ($(G\backslash G^2)\backslash G^3$).}
\label{fig:graphdefs}
\end{center}
\end{figure}

Let $T \geq 1$ and we define
\begin{align*}
\mathcal G & := \left\{\, G(\{w\}) \mid  \w \in [3]^T_{NL}, G(\{w\}) \text{ has only one type of two-cycle }\,\right\},\\
\mathcal G_{m,n} & := \left\{\, G \in \mathcal G \mid \frac{|G^2|}{2} = m,\, \frac{|(G \backslash G^2)^3|}{3} = n \,\right\}.
\end{align*}

Notice that the graph $G$ in Figure \ref{fig:graphdefs} is in $\mathcal G$
since it has only one type of two-cycle. Moreover, $G$ is contained in $\mathcal G_{2,1}$.

\begin{rem}\label{rem:threecycles4twocycles}
If a multigraph $G$ on three vertices and $T-1$ edges with no self-loops has only
one type of two-cycle, then every three-cycle must have the same orientation.
See Figure \ref{fig:oppthreecycles}.
\end{rem}

\begin{rem}\label{rem:finitesize}
Note that 
\begin{equation*}
| \mathcal G_{m,n} | \leq 18
\end{equation*}
for any $m$,$n$, and $T$. There are $m$ two-cycles of the same type, hence
three ways to place them. The $n$ three-cycles must be the same orientation by
Remark \ref{rem:threecycles4twocycles}, hence two ways to place them. Finally,
the remaining one or two edges must be placed in the same orientation as the
three-cycles, hence three ways to place them.
\end{rem}

\begin{figure}
\begin{center}
\includegraphics{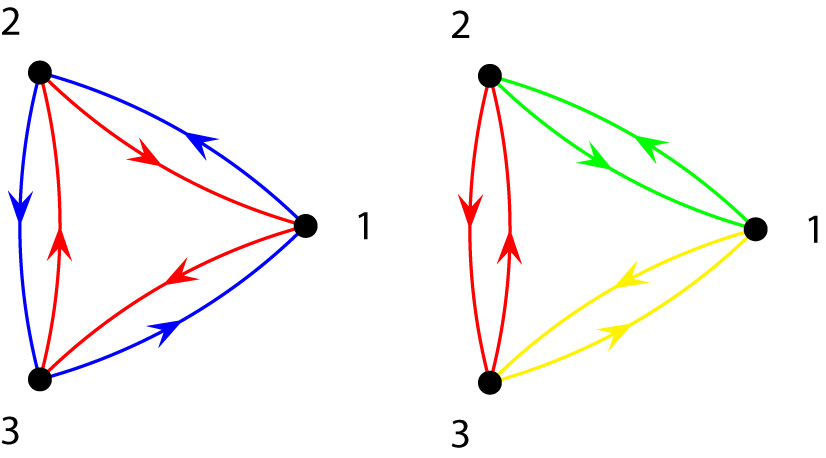}
\caption{The red and blue three-cycles in the multigraph on the left have opposite orientation, implying there are multiple types of two cycles, red, green and yellow in the multigraph on the right. Hence, this is not a vertex of $\Poly{4}$.}
\label{fig:oppthreecycles}
\end{center}
\end{figure}


\begin{lem}
\label{lem:verttwocycle}
Let $S=3$ and $T \geq 1$ and $\ve x \in \Poly{4}$ be a vertex. If $i \rightarrow j \rightarrow i$ and
$k \rightarrow l \rightarrow k$ are two-cycles in $G(\ve x)$, then they consist of the same edges. That
is, $\{il\} = \{kl\}$.
\end{lem}
\begin{proof}
We will prove the contrapositive. Suppose $\ve x \in \Poly{4} \cap \Z^6$ where
$i \rightarrow j \rightarrow i$ and $k \rightarrow l \rightarrow k$ are
two-cycles in $G(\ve x)$ such that they do not consist of the same edges. Then
$x_{ij} > 1$, $x_{ji} > 1$, $x_{kl} > 1$ and $x_{kl} > 1$. Moreoever $\{ij\}
\cap \{kl\} = \emptyset$. Let $\ve y = \ve x + \ve e_{kl} + \ve e_{lk} - \ve
e_{ij} - \ve e_{ji}$ and $\ve z =  \ve x + \ve e_{ij} + \ve
e_{ji} - \ve e_{kl} - \ve e_{lk}$. Note that, by Remark
\ref{rem:eulpath}, we must have $\ve y, \ve z \in \Poly{4}$ since
we are only removing and adding two-cycles. Then $\ve x = \frac{1}{2} \ve y + \frac{1}{2} \ve z$.
\end{proof}

Lemma \ref{lem:verttwocycle} can be proved by also considering convex
combinations of multigraphs. See Figure \ref{fig:lincombexone}. 
\begin{figure}
\begin{center}
\scalebox{0.7}{
\includegraphics{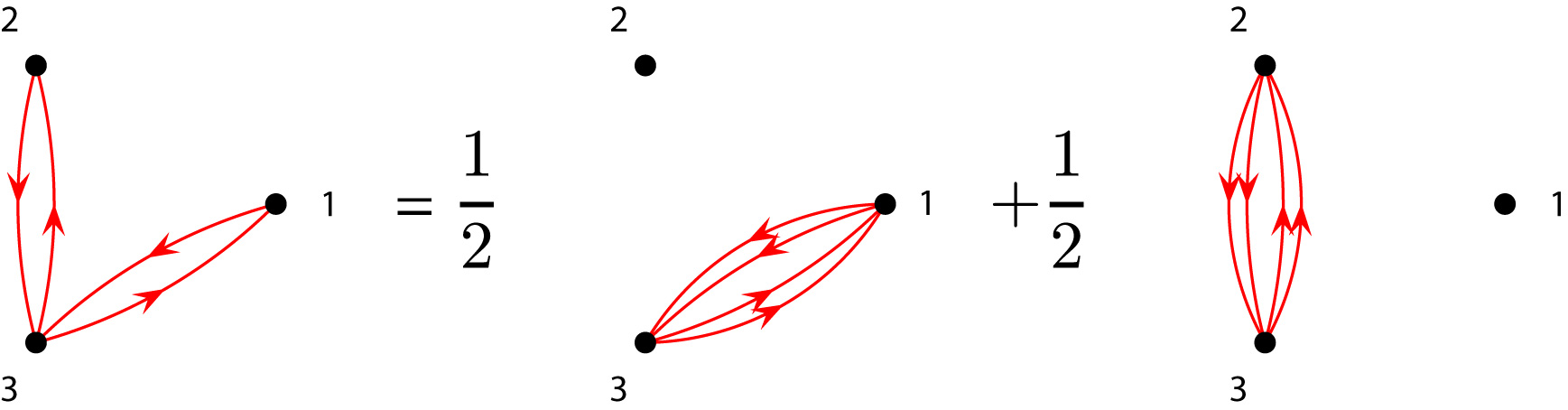}
}
\caption{On the left is the multigraph for $[0 \; 1 \; 0 \; 1 \; 1 \; 1 ]^\top \in \A{4}_{3,5}$
which can be written as a linear combination of vectors $[0 \; 2 \; 0 \; 0 \; 2
\; 0 ]^\top  \in \A{4}_{3,5}$ and $[0 \; 0 \; 0 \; 2 \; 0 \; 2 ]^\top \in \A{4}_{3,5}$.}
\label{fig:lincombexone}
\end{center}
\end{figure}

By definition of the convex hull, the vertices of $\Poly{4}$ will be contained
in the columns of $\A{4}$.  By Lemma \ref{lem:verttwocycle}, for $S=3$, the
vertices of $\Poly{4}$ will be contained in $\mathcal G$, the set of directed
multigraphs on three vertices that have only one type of two-cycle.  It is not
difficult to see that $ \mathcal G = \bigcup_{n,m \in \Z_{\geq 0}} \mathcal
G_{n,m}$.  Note that $\mathcal G_{m,n}$ is non-empty depending on $T$, $m$ and
$n$.

For $t\in\R$, let 
\begin{equation*}
f_T(t) = \begin{cases}
 \left\lfloor\frac{(T-1 - 2t)}{3}\right\rfloor   & \text{if }  0 \leq 2t \leq T-1, \\
  \ \ 0     & \text{otherwise}.
\end{cases}
\end{equation*}
\begin{proposition}
Let $S=3$ and $T \geq 1$. Then $\mathcal G_{m,f(m)} \neq \emptyset$ for $ 0
\leq 2m \leq T-1$, else $\mathcal G_{m,n} = \emptyset$.
\end{proposition}

Therefore, if $p = \lfloor (T-1)/2 \rfloor$, we can write $\mathcal G =
\mathcal G_{0,f(0)} \cup \cdots \cup \mathcal G_{p,f(p)} $. The main idea behind
Theorem \ref{thm:finitevert} is that for $S=3$, all graphs (vectors) in
$\mathcal G_{m,n}$ are not vertices for many $(m,n)$.

\begin{thm}
\label{thm:finitevert}
Let $S=3$. The number of vertices of $\Poly{4}$ is bounded by some constant $C$ which does not depend on $T$. 
\end{thm}
\begin{proof}
Let $T \geq 13$ and $p = \lfloor (T-1)/2 \rfloor$. We have $\mathcal G =
\mathcal \mathcal G_{0,f(0)} \cup \cdots \cup \mathcal G_{p,f(p)}$. We now claim for $ 3
\leq q \leq p-3$, that every graph(vector) in $\mathcal G_{q,f(q)}$ is not a
vertex. Let $\ve x \in \mathcal G_{q,f(q)}$ for $ 3 \leq q \leq p+3$. Note that
$\frac{1}{2}f(q-3) + \frac{1}{2}f(q+3) = f(q)$. Let $\ve y$ be derived from $\ve x$
where two of the three-cycles are removed and three two-cycles are added such
that $\ve y \in \mathcal G_{q-3,f(q-3)}$. Similarly let $\ve z$ be derived from $\ve
x$ where three two-cycles are removed and two of the three-cycles are added
such that $\ve z \in \mathcal G_{q+3,f(q+3)}$. Finally, $\ve x = \frac{1}{2}\ve
y + \frac{1}{2} \ve z$. We illustrate this construction in Figure \ref{fig:thmex}.

Thus, the vertices of $\Poly{4}$ must be contained in $\mathcal G_{0,f(0)}\cup
\mathcal G_{1,f(1)}\cup \mathcal G_{2,f(2)}\cup \mathcal G_{p-2,f(p-2)}\cup \mathcal G_{p-1,f(p-1)}\cup
\mathcal G_{p,f(p)}$, which is bounded for all $T$. Note that each set of graphs
in the above union is finite by definition.
\end{proof}

\begin{figure}
\begin{center}
\scalebox{0.7}{
\includegraphics{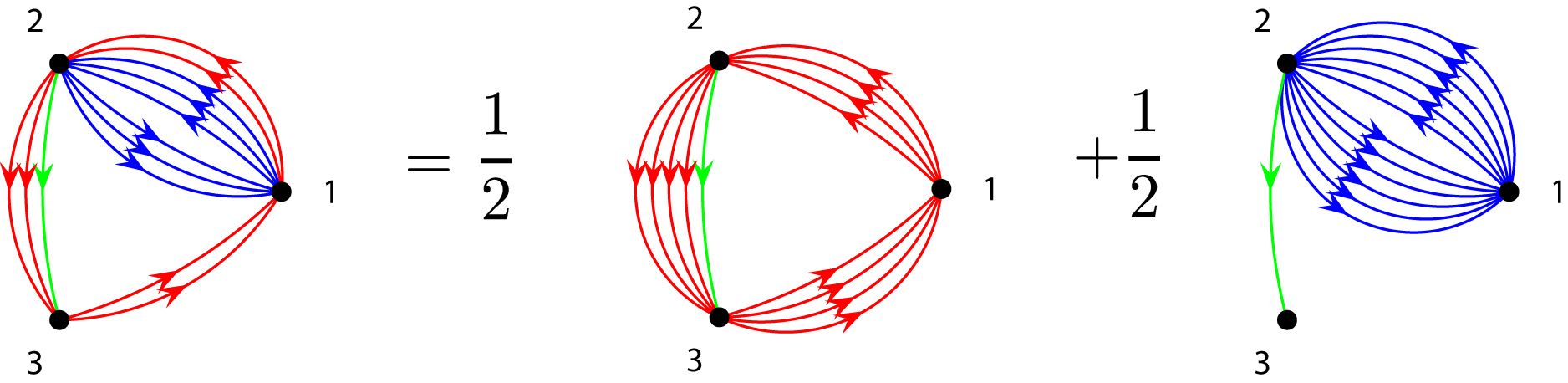}
}
\caption{The graph $G$ on the left (from word 1212121231231) is contained in $\mathcal G_{3,2}$, and by Theorem \ref{thm:finitevert}, can be written as $\frac{1}{2}G(\{23123123123123\}) + \frac{1}{2}G(\{212121212123\})$. Hence, $G$ is not a vertex of $\Poly{4}$.}
\label{fig:thmex}
\end{center}
\end{figure}

\section{Computational Results}\label{sec:computations}
Here we give two tables listing the f-vector and number of Hilbert basis
elements for $S=3$ and Model \eqref{model3} and \eqref{model4}. The Hilbert
basis and supporting hyperplanes were computed with {\tt normaliz}
\cite{normaliz}, and the f-vectors were computed using {\tt Polymake}
\cite{polymake}. For Model \eqref{model3} we computed the f-vector and Hilbert
basis for $S=3$ and $T=1,\ldots,9$, shown in Figure \ref{tab:hb_fvec_model3}. For Model
\eqref{model4} we computed the f-vector and Hilbert basis for $S=3$ and
$T=1,\ldots,15$, shown in Figure \ref{tab:hb_fvec_model4}. 
The supporting hyperplanes of $\Cone{3}$ and $\Cone{4}$ for $S=3$ computed by
{\tt normaliz}\cite{normaliz} are given in the Appendix.

\begin{table}
\begin{tabular}{l|r|rrrrrrr}
T & $\#$HB & $f_0$ & $f_1$ & $f_2$ & $f_3$ & $f_4$ & $f_5$ & $f_6$\\
\hline 
    4    &   24 &    24 &   156 &   434 &   606 &   444 &   162 &    24   \\
    5    &   39 &    36 &   249 &   671 &   891 &   615 &   210 &    30   \\
    6    &   60 &    42 &   276 &   689 &   837 &   528 &   168 &    24   \\
    7    &   87 &    54 &   351 &   860 &  1020 &   633 &   204 &    30   \\
    8    &  120 &    60 &   372 &   851 &   939 &   546 &   168 &    24   \\
    9    &  162 &    72 &   435 &   968 &  1062 &   633 &   204 &    30
\end{tabular}
\caption{
The number of Hilbert basis elements ({\tt normaliz}) and f-vectors ({\tt
Polymake}) for Model \eqref{model3} where $S=3$. The running time of {\tt
normaliz} was under two seconds for all data sets.
}
\label{tab:hb_fvec_model3}
\end{table}

\begin{table}
\begin{tabular}{l|r|rrrrr}
T & $\#$HB & $f_0$ & $f_1$ & $f_2$ & $f_3$ & $f_4$\\
\hline 
   4 & 20  &  20 &   69 &   90 &   51 &   12\\
   5 & 30  &  27 &  114 &  167 &  102 &   24\\
   6 & 48  &  24 &  111 &  176 &  111 &   24\\
   7 & 66  &  41 &  144 &  189 &  108 &   24\\
   8 & 96  &  42 &  171 &  230 &  123 &   24\\
   9 & 123 &  45 &  186 &  245 &  126 &   24\\
  10 & 166 &  56 &  201 &  252 &  129 &   24\\
  11 & 207 &  63 &  216 &  257 &  126 &   24\\
  12 & 264 &  54 &  189 &  236 &  123 &   24\\
  13 & 320 &  77 &  246 &  279 &  132 &   24\\
  14 & 396 &  54 &  189 &  236 &  123 &   24\\
  15 & 468 &  63 &  216 &  257 &  126 &   24
\end{tabular}
\caption{
The number of Hilbert basis elements ({\tt normaliz}) and f-vectors ({\tt Polymake}) for Model \eqref{model4} where $S=3$. The running time of {\tt normaliz} was under two seconds for all data sets.
}
\label{tab:hb_fvec_model4}
\end{table}

All supplementary material can be found at \url{http://www.davidhaws.net/Projects/ToricMarkovChain/}.  
Software to draw state graphs and move graphs can be found at \url{https://github.com/dchaws/DrawStateMoveGraphs}.
Software to generate all words, all words with no self-loops and the design
matrices can be found at \url{https://github.com/dchaws/GenWordsTrans}.

\section{Conclusions and Open Problems}\label{sec:openproblems}

One notices that the set of columns is a graded set since there exists
$\ve w \in \Q^{ S^2}$ such that ${\bf a_i} \cdot \ve w = 1$  by Lemma 4.14
in Sturmfels\cite{sturmfels1996}.

One tool is coming from Theorem 13.14 in Sturmfels\cite{sturmfels1996}.
\begin{thm}[Theorem 13.14 in Sturmfels\cite{sturmfels1996}]\label{normal_thm}
Let $\cA \subset \Z^d$ be a graded set such that the semigroup generated
by the elements in $A$ is normal.  Then the toric ideal $I_A$ associate with
the set $A$ is generated by homogeneous binomials of degree at most $d$.
\end{thm}

Using Theorem \ref{thm:finitevert}, Conjecture
\ref{loopless_noinitial_normal2} and Theorem 
\ref{normal_thm}, one can prove the following theorem:

\begin{conj}
We consider Model \eqref{model4}. Then 
for $S = 3$ and for any $T \geq 4$, a Markov basis for the toric ideal
$I_{\A{4}}$ consists of binomials of degree less than or equal to
$d= 6$.  Moreover, there are only finitely many moves up to a
certain shift equivalence relation. 
\end{conj}

On the experimentations we ran, we found evidence that more should be true.

\begin{conj}\label{conj:MBbound}
Fix $S\geq 3$; then, for every $T\geq 4$, there is a Markov basis for the toric
ideal $I_{\A{4}}$ consisting of binomials of degree at most $S-1$, and there is
a Gr\"obner basis with respect to some term ordering consisting of binomials of
degree at most $S$.
\end{conj}

Despite the computational limitations (the number of generators grows
exponentially when $T$ grows,) we were able to test Conjecture
\ref{conj:MBbound} using {\tt 4ti2}\cite{4ti2} for the following cases:

\begin{table}
\begin{tabular}{c|c|c|c|}
\cline{2-4}
& \multicolumn{3}{|c|}{S} \\ \cline{1-4}
\multicolumn{1}{|c|}{$T=3$} & 4 & 5 & 6 \\ \hline
\multicolumn{1}{|c|}{$T=4$} & 4 & 5 & \\ \hline
\end{tabular}
\caption{
Cases where Conjecture \ref{conj:MBbound} was tested
}
\end{table}

\appendix{Supporting Hyperplanes}
In this appendix, we present the supporting hyperplanes of $\Cone{4}$ for $S=3$ computed by
{\tt normaliz}\cite{normaliz}. Hyperplanes are given by column vectors
$[c_{12} \; c_{13} \; c_{21} \; c_{23} \; c_{31} \; c_{32}]^\top$ where
$c_{12}x_{12} + c_{13}x_{13} + c_{21}x_{21} + c_{23}x_{23} + c_{31}x_{31} +
c_{32}x_{32}  \geq 0$. For all cases computed, the non-negativity constraints
were given as hyperplanes and are not included for brevity.

\begin{gather*}
T=4, \\
\left[
\begin{array}{rrrrrr}
      2   &   2   &   2   &  -1   &  -1   &  -1  \\
     -1   &   2   &   2   &  -1   &  -1   &   2  \\
     -1   &  -1   &  -1   &   2   &   2   &   2  \\
     -1   &  -1   &   2   &  -1   &   2   &   2  \\
      2   &  -1   &  -1   &   2   &   2   &  -1  \\
      2   &   2   &  -1   &   2   &  -1   &  -1
\end{array}
\right]
\end{gather*}

\begin{gather*}
T=5, \\
\left[
\begin{array}{rrrrrr|rrrrrr|rrrrrr}
 5  & 5  & 1  & 1  & -3  & -3  & 1  & 1  & 1  & -1  & -1  & 1  & 3  & 3  & 3  & -1  & -1  & -1\\
 5  & 1  & -3  & 5  & -3  & 1  & 1  & 1  & 1  & -1  & 1  & -1  & -1  & 3  & 3  & -1  & -1  & 3\\
 -3  & -3  & 1  & 1  & 5  & 5  & -1  & -1  & 1  & 1  & 1  & 1  & -1  & -1  & -1  & 3  & 3  & 3\\
 1  & -3  & -3  & 5  & 1  & 5  & -1  & 1  & 1  & 1  & 1  & -1  & -1  & -1  & 3  & -1  & 3  & 3\\
 -3  & 1  & 5  & -3  & 5  & 1  & 1  & -1  & -1  & 1  & 1  & 1  & 3  & -1  & -1  & 3  & 3  & -1\\
 1  & 5  & 5  & -3  & 1  & -3  & 1  & 1  & -1  & 1  & -1  & 1  & 3  & 3  & -1  & 3  & -1  & -1\\
\end{array}
\right]
\end{gather*}

\begin{gather*}
T=6, \\
\left[
\begin{array}{rrrrrr|rrrrrr|rrrrrr}
6  & 6  & 1  & 1  & -4  & -4  & 9  & 4  & 4  & -1  & -1  & -6  & 8  & 3  & 3  & -2  & -2  & -2\\
6  & 1  & -4  & 6  & -4  & 1  & 4  & -1  & 9  & -6  & 4  & -1  & 3  & -2  & 8  & -2  & -2  & 3\\
-4  & -4  & 1  & 1  & 6  & 6  & -6  & -1  & -1  & 4  & 4  & 9  & -2  & -2  & -2  & 8  & 3  & 3\\
1  & -4  & -4  & 6  & 1  & 6  & -1  & -6  & 4  & -1  & 9  & 4  & -2  & -2  & 3  & 3  & -2  & 8\\
-4  & 1  & 6  & -4  & 6  & 1  & -1  & 4  & -6  & 9  & -1  & 4  & -2  & 3  & -2  & 3  & 8  & -2\\
1  & 6  & 6  & -4  & 1  & -4  & 4  & 9  & -1  & 4  & -6  & -1  & 3  & 8  & -2  & -2  & 3  & -2\\
\end{array}
\right]
\end{gather*}

\begin{gather*}
T=7, \\
\left[
\begin{array}{rrrrrr|rrrrrr|rrrrrr}
7  & 7  & 1  & 1  & -5  & -5  & 1  & 1  & 1  & 1  & -1  & -1  & 2  & 2  & 2  & -1  & -1  & -1\\
7  & 1  & -5  & 7  & -5  & 1  & -1  & 1  & 1  & 1  & -1  & 1  & -1  & 2  & 2  & -1  & -1  & 2\\
-5  & -5  & 1  & 1  & 7  & 7  & 1  & -1  & -1  & 1  & 1  & 1  & -1  & -1  & -1  & 2  & 2  & 2\\
1  & -5  & -5  & 7  & 1  & 7  & -1  & -1  & 1  & 1  & 1  & 1  & -1  & -1  & 2  & -1  & 2  & 2\\
-5  & 1  & 7  & -5  & 7  & 1  & 1  & 1  & -1  & -1  & 1  & 1  & 2  & -1  & -1  & 2  & 2  & -1\\
1  & 7  & 7  & -5  & 1  & -5  & 1  & 1  & 1  & -1  & 1  & -1  & 2  & 2  & -1  & 2  & -1  & -1\\
\end{array}
\right]
\end{gather*}

\begin{gather*}
T=8 \\
\left[
\begin{array}{rrrrrr|rrrrrr|rrrrrr}
8  & 8  & 1  & 1  & -6  & -6  & 11  & 4  & 4  & -3  & -3  & -3  & 5  & 5  & 5  & -2  & -2  & -2\\
8  & 1  & -6  & 8  & -6  & 1  & 4  & 11  & -3  & -3  & -3  & 4  & -2  & 5  & 5  & -2  & -2  & 5\\
-6  & -6  & 1  & 1  & 8  & 8  & -3  & -3  & -3  & 11  & 4  & 4  & -2  & -2  & -2  & 5  & 5  & 5\\
1  & -6  & -6  & 8  & 1  & 8  & -3  & 4  & -3  & 4  & -3  & 11  & -2  & -2  & 5  & -2  & 5  & 5\\
-6  & 1  & 8  & -6  & 8  & 1  & -3  & -3  & 4  & 4  & 11  & -3  & 5  & -2  & -2  & 5  & 5  & -2\\
1  & 8  & 8  & -6  & 1  & -6  & 4  & -3  & 11  & -3  & 4  & -3  & 5  & 5  & -2  & 5  & -2  & -2\\
\end{array}
\right]
\end{gather*}

\begin{gather*}
T=9, \\
\left[
\begin{array}{rrrrrr|rrrrrr|rrrrrr}
9  & 9  & 1  & 1  & -7  & -7  & 1  & 1  & 1  & 1  & -1  & -1  & 7  & 3  & 3  & -1  & -1  & -5\\
9  & 1  & -7  & 9  & -7  & 1  & -1  & 1  & 1  & 1  & -1  & 1  & 3  & -1  & 7  & -5  & 3  & -1\\
-7  & -7  & 1  & 1  & 9  & 9  & 1  & -1  & -1  & 1  & 1  & 1  & -5  & -1  & -1  & 3  & 3  & 7\\
1  & -7  & -7  & 9  & 1  & 9  & -1  & -1  & 1  & 1  & 1  & 1  & -1  & -5  & 3  & -1  & 7  & 3\\
-7  & 1  & 9  & -7  & 9  & 1  & 1  & 1  & -1  & -1  & 1  & 1  & -1  & 3  & -5  & 7  & -1  & 3\\
1  & 9  & 9  & -7  & 1  & -7  & 1  & 1  & 1  & -1  & 1  & -1  & 3  & 7  & -1  & 3  & -5  & -1\\
\end{array}
\right]
\end{gather*}

\begin{gather*}
T=10, \\
\left[
\begin{array}{rrrrrr|rrrrrr|rrrrrr}
10  & 10  & 1  & 1  & -8  & -8  & 14  & 5  & 5  & -4  & -4  & -4  & 2  & 2  & 2  & -1  & -1  & -1\\
10  & 1  & 10  & -8  & -8  & 1  & 5  & 14  & -4  & -4  & -4  & 5  & -1  & 2  & 2  & -1  & -1  & 2\\
-8  & -8  & 1  & 1  & 10  & 10  & -4  & -4  & -4  & 14  & 5  & 5  & -1  & -1  & -1  & 2  & 2  & 2\\
1  & -8  & 10  & -8  & 1  & 10  & -4  & 5  & -4  & 5  & -4  & 14  & -1  & -1  & 2  & -1  & 2  & 2\\
-8  & 1  & -8  & 10  & 10  & 1  & -4  & -4  & 5  & 5  & 14  & -4  & 2  & -1  & -1  & 2  & 2  & -1\\
1  & 10  & -8  & 10  & 1  & -8  & 5  & -4  & 14  & -4  & 5  & -4  & 2  & 2  & -1  & 2  & -1  & -1\\
\end{array}
\right]
\end{gather*}

\begin{gather*}
T=11, \\
\left[
\begin{array}{rrrrrr|rrrrrr|rrrrrr}
11  &  11  &  1  &  1  &  -9  &  -9  &  1  &  1  &  1  &  1  &  -1  &  -1  &  7  &  7  &  7  &  -3  &  -3  &  -3\\  
11  &  1  &  11  &  -9  &  -9  &  1  &  -1  &  1  &  1  &  1  &  -1  &  1  &  -3  &  7  &  7  &  -3  &  -3  &  7\\  
-9  &  -9  &  1  &  1  &  11  &  11  &  1  &  -1  &  -1  &  1  &  1  &  1  &  -3  &  -3  &  -3  &  7  &  7  &  7\\  
1  &  -9  &  11  &  -9  &  1  &  11  &  -1  &  -1  &  1  &  1  &  1  &  1  &  -3  &  -3  &  7  &  -3  &  7  &  7\\  
-9  &  1  &  -9  &  11  &  11  &  1  &  1  &  1  &  -1  &  -1  &  1  &  1  &  7  &  -3  &  -3  &  7  &  7  &  -3\\  
1  &  11  &  -9  &  11  &  1  &  -9  &  1  &  1  &  1  &  -1  &  1  &  -1  &  7  &  7  &  -3  &  7  &  -3  &  -3  
\end{array}
\right]
\end{gather*}

\begin{gather*}
T=12, \\
\left[ \begin{array}{rrrrrr|rrrrrr|r}
  12   &  12   &   1   &   1   & -10   & -10   &  17   &   6   &   6   &  -5   &  -5   &  -5   & \cdots \\
  12   &   1   & -10   &  12   & -10   &   1   &   6   &  17   &  -5   &  -5   &  -5   &   6   & \cdots \\
 -10   & -10   &   1   &   1   &  12   &  12   &  -5   &  -5   &  -5   &  17   &   6   &   6   & \cdots \\
   1   & -10   & -10   &  12   &   1   &  12   &  -5   &   6   &  -5   &   6   &  -5   &  17   & \cdots \\
 -10   &   1   &  12   & -10   &  12   &   1   &  -5   &  -5   &   6   &   6   &  17   &  -5   & \cdots \\
   1   &  12   &  12   & -10   &   1   & -10   &   6   &  -5   &  17   &  -5   &   6   &  -5   & \cdots 
\end{array} \right. \\
\left. \qquad \qquad \begin{array}{r|rrrrrr}
\cdots &  19   &   8   &   8   &  -3   &  -3   & -14  \\
\cdots &   8   &  19   &  -3   & -14   &   8   &  -3  \\
\cdots & -14   &  -3   &  -3   &   8   &   8   &  19  \\
\cdots &  -3   &   8   & -14   &  -3   &  19   &   8  \\
\cdots &  -3   & -14   &   8   &  19   &  -3   &   8  \\
\cdots &   8   &  -3   &  19   &   8   & -14   &  -3
\end{array}  \right]
\end{gather*}

\begin{gather*}
T=13, \\
\left[
\begin{array}{rrrrrr|rrrrrr|rrrrrr}
13  &  13  &  1  &  1  &  -11  &  -11  &  1  &  1  &  1  &  1  &  -1  &  -1  &  2  &  2  &  2  &  -1  &  -1  &  -1\\    
13  &  1  &  -11  &  13  &  -11  &  1  &  -1  &  1  &  1  &  1  &  -1  &  1  &  -1  &  2  &  2  &  -1  &  -1  &  2\\    
-11  &  -11  &  1  &  1  &  13  &  13  &  1  &  -1  &  -1  &  1  &  1  &  1  &  -1  &  -1  &  -1  &  2  &  2  &  2\\    
1  &  -11  &  -11  &  13  &  1  &  13  &  -1  &  -1  &  1  &  1  &  1  &  1  &  -1  &  -1  &  2  &  -1  &  2  &  2\\    
-11  &  1  &  13  &  -11  &  13  &  1  &  1  &  1  &  -1  &  -1  &  1  &  1  &  2  &  -1  &  -1  &  2  &  2  &  -1\\    
1  &  13  &  13  &  -11  &  1  &  -11  &  1  &  1  &  1  &  -1  &  1  &  -1  &  2  &  2  &  -1  &  2  &  -1  &  -1    
\end{array}
\right]
\end{gather*}

\begin{gather*}
T=14, \\
\left[
\begin{array}{rrrrrr|rrrrrr|rrrrrr}
14 & 14 & 1 & 1 & -12 & -12 & 20 & 7 & 7 & -6 & -6 & -6 & 9 & 9 & 9 & -4 & -4 & -4\\ 
14 & 1 & -12 & 14 & -12 & 1 & 7 & 20 & -6 & -6 & -6 & 7 & -4 & 9 & 9 & -4 & -4 & 9\\  
-12 & -12 & 1 & 1 & 14 & 14 & -6 & -6 & -6 & 20 & 7 & 7 & -4 & -4 & -4 & 9 & 9 & 9\\  
1 & -12 & -12 & 14 & 1 & 14 & -6 & 7 & -6 & 7 & -6 & 20 & -4 & -4 & 9 & -4 & 9 & 9\\  
-12 & 1 & 14 & -12 & 14 & 1 & -6 & -6 & 7 & 7 & 20 & -6 & 9 & -4 & -4 & 9 & 9 & -4\\  
1 & 14 & 14 & -12 & 1 & -12 & 7 & -6 & 20 & -6 & 7 & -6 & 9 & 9 & -4 & 9 & -4 & -4  
\end{array}
\right]
\end{gather*}

\begin{gather*}
T=15, \\
\left[
\begin{array}{rrrrrr|rrrrrr|rrrrrr}
15 & 15 & 1 & 1 & -13 & -13 & 1 & 1 & 1 & 1 & -1 & -1 & 12 & 5 & 5 & -2 & -2 & -9\\  
15 & 1 & -13 & 15 & -13 & 1 & -1 & 1 & 1 & 1 & -1 & 1 & 5 & 12 & -2 & -9 & 5 & -2\\  
-13 & -13 & 1 & 1 & 15 & 15 & 1 & -1 & -1 & 1 & 1 & 1 & -9 & -2 & -2 & 5 & 5 & 12\\  
1 & -13 & -13 & 15 & 1 & 15 & -1 & -1 & 1 & 1 & 1 & 1 & -2 & 5 & -9 & -2 & 12 & 5\\  
-13 & 1 & 15 & -13 & 15 & 1 & 1 & 1 & -1 & -1 & 1 & 1 & -2 & -9 & 5 & 12 & -2 & 5\\  
1 & 15 & 15 & -13 & 1 & -13 & 1 & 1 & 1 & -1 & 1 & -1 & 5 & -2 & 12 & 5 & -9 & -2  
\end{array}
\right]
\end{gather*}

In what follows, we present the supporting hyperplanes of $\Cone{3}$ for $S=3$ computed by
{\tt normaliz}\cite{normaliz}. Hyperplanes are given by column vectors
$[c_{1} \; c_{2} \; c_{3} \; c_{12} \; c_{13} \; c_{21} \; c_{23} \; c_{31} \;
c_{32}]^\top$ where $c_1x_1 + c_2x_2 + c_3x_3 + c_{12}x_{12} + c_{13}x_{13} +
c_{21}x_{21} + c_{23}x_{23} + c_{31}x_{31} + c_{32}x_{32}  \geq 0$.  For all
cases computed, the non-negativity constraints were given as hyperplanes and
are not included for brevity. The number of hyperplanes alternates between
$22$ and $30$.

\begin{gather*}
T=4, \\
\left[
\begin{array}{rrrrrrrrrrrrrrrr}
 -3  & -2  & -1  & -1  & -1  & -1  &  3  &  3  &  2  &  2  &  2  &  2  &  2  &  1  &  1  &  1 \\ 
 -3  & -2  & -2  & -1  & -1  & -1  &  4  &  3  &  2  &  2  &  2  &  2  &  1  &  2  &  2  &  0 \\ 
 -2  & -1  & -1  & -1  & -1  & -1  &  3  &  3  &  2  &  2  &  2  &  1  &  1  &  2  &  1  &  0 \\ 
  1  &  1  &  0  &  1  &  0  &  0  &  0  & -1  & -1  &  0  &  0  &  0  & -1  &  0  &  0  & -1 \\ 
  2  &  1  &  1  &  1  &  1  &  0  & -1  & -1  & -1  & -1  &  0  & -1  & -1  &  0  &  0  & -1 \\ 
  1  &  1  &  1  &  0  &  1  &  1  & -2  & -1  &  0  & -1  & -1  &  0  &  0  & -1  & -1  &  1 \\ 
  2  &  1  &  1  &  1  &  1  &  1  & -2  & -1  & -1  & -1  & -1  & -1  &  0  &  0  & -1  &  0 \\ 
  0  &  0  &  1  &  0  &  0  &  1  & -1  & -1  &  0  &  0  & -1  &  0  &  0  & -1  &  0  &  1 \\ 
  0  &  0  &  0  &  0  &  0  &  0  &  0  &  0  &  0  &  0  &  0  &  0  &  0  &  0  &  0  &  0  
\end{array}
\right]
\end{gather*}

\begin{gather*}
T=5, \\
\left[
\begin{array}{rrrrrrrrrrrrr}
 -4  & -3  & -2  & -2  & -2  & -2  & -1  & -1  &  6  &  4  &  4  & \cdots \\ 
 -4  & -2  & -3  & -2  & -2  & -1  & -2  & -2  &  6  &  5  &  4  & \cdots \\ 
 -3  & -2  & -2  & -2  & -2  & -1  & -1  & -1  &  5  &  4  &  4  & \cdots \\ 
  1  &  1  &  1  &  1  &  0  &  1  &  0  &  0  & -1  &  0  & -1  & \cdots \\ 
  2  &  2  &  1  &  1  &  1  &  1  &  1  &  0  & -1  & -1  & -1  & \cdots \\ 
  1  &  1  &  1  &  1  &  1  &  0  &  1  &  1  & -2  & -2  & -1  & \cdots \\ 
  2  &  1  &  2  &  1  &  1  &  1  &  1  &  1  & -2  & -2  & -1  & \cdots \\ 
  0  &  0  &  0  &  0  &  1  &  0  &  0  &  1  & -1  & -1  & -1  & \cdots \\ 
  0  &  0  &  0  &  0  &  0  &  0  &  0  &  0  &  0  &  0  &  0  & \cdots 
\end{array} \right. \\
\left. \qquad \qquad \begin{array}{rrrrrrrrrrrr}
\cdots &  3  &  3  &  2  &  2  &  2  &  2  &  2  &  2  &  2  &  1  &  1 \\ 
\cdots &  3  &  3  &  3  &  2  &  2  &  2  &  2  &  2  &  1  &  2  &  0 \\ 
\cdots &  2  &  2  &  3  &  2  &  2  &  2  &  2  &  1  &  2  &  2  &  0 \\ 
\cdots & -1  &  0  &  0  & -1  &  0  &  0  &  0  & -1  & -1  &  1  & -1 \\ 
\cdots & -1  & -1  &  0  & -1  & -1  &  0  &  0  & -1  & -1  &  0  & -1 \\ 
\cdots &  0  & -1  & -1  &  0  &  0  & -1  & -1  &  0  &  1  & -1  &  1 \\ 
\cdots & -1  & -1  & -1  &  0  & -1  & -1  &  0  &  0  &  0  & -1  &  0 \\ 
\cdots &  0  &  0  & -1  &  0  &  0  &  0  & -1  &  1  &  0  &  0  &  1 \\ 
\cdots &  0  &  0  &  0  &  0  &  0  &  0  &  0  &  0  &  0  &  0  &  0  
\end{array}
\right]
\end{gather*}

\begin{gather*}
T=6, \\
\left[
\begin{array}{rrrrrrrrrrrrrrrr}
 -5  & -3  & -2  & -2  & -2  & -1  &  5  &  5  &  4  &  3  &  3  &  3  &  3  &  2  &  2  &  1 \\ 
 -5  & -3  & -3  & -2  & -2  & -2  &  6  &  5  &  3  &  4  &  4  &  3  &  2  &  3  &  3  &  0 \\ 
 -4  & -2  & -2  & -1  & -1  & -2  &  5  &  5  &  3  &  3  &  3  &  2  &  2  &  3  &  2  &  0 \\ 
  1  &  1  &  0  &  1  &  0  &  0  &  0  & -1  & -1  &  0  &  0  &  0  & -1  &  0  &  0  & -1 \\ 
  2  &  1  &  1  &  1  &  1  &  0  & -1  & -1  & -1  & -1  &  0  & -1  & -1  &  0  &  0  & -1 \\ 
  1  &  1  &  1  &  0  &  1  &  1  & -2  & -1  &  0  & -1  & -1  &  0  &  0  & -1  & -1  &  1 \\ 
  2  &  1  &  1  &  1  &  1  &  1  & -2  & -1  & -1  & -1  & -1  & -1  &  0  &  0  & -1  &  0 \\ 
  0  &  0  &  1  &  0  &  0  &  1  & -1  & -1  &  0  &  0  & -1  &  0  &  0  & -1  &  0  &  1 \\ 
  0  &  0  &  0  &  0  &  0  &  0  &  0  &  0  &  0  &  0  &  0  &  0  &  0  &  0  &  0  &  0  
\end{array}
\right]
\end{gather*}

\begin{gather*}
T=7, \\
\left[
\begin{array}{rrrrrrrrrrrr}
 -6  & -4  & -3  & -3  & -3  & -2  & -2  & -2  &  9  &  6  &  6  & \cdots \\ 
 -6  & -3  & -4  & -3  & -3  & -2  & -2  & -2  &  9  &  7  &  6  & \cdots \\ 
 -5  & -3  & -3  & -3  & -3  & -2  & -2  & -2  &  8  &  6  &  6  & \cdots \\ 
  1  &  1  &  1  &  1  &  0  &  1  &  0  &  0  & -1  &  0  & -1  & \cdots \\ 
  2  &  2  &  1  &  1  &  1  &  1  &  1  &  0  & -1  & -1  & -1  & \cdots \\ 
  1  &  1  &  1  &  1  &  1  &  0  &  1  &  1  & -2  & -2  & -1  & \cdots \\ 
  2  &  1  &  2  &  1  &  1  &  1  &  1  &  1  & -2  & -2  & -1  & \cdots \\ 
  0  &  0  &  0  &  0  &  1  &  0  &  0  &  1  & -1  & -1  & -1  & \cdots \\ 
  0  &  0  &  0  &  0  &  0  &  0  &  0  &  0  &  0  &  0  &  0  & \cdots
\end{array} \right. \\
\left. \qquad \qquad \begin{array}{rrrrrrrrrrrr}
\cdots &  4  &  4  &  4  &  3  &  3  &  3  &  3  &  3  &  3  &  2  &  1 \\ 
\cdots &  4  &  4  &  4  &  3  &  3  &  3  &  3  &  3  &  2  &  3  &  0 \\ 
\cdots &  4  &  4  &  4  &  3  &  3  &  3  &  3  &  2  &  3  &  3  &  0 \\ 
\cdots & -1  &  0  &  0  & -1  &  0  &  0  &  0  & -1  & -1  &  1  & -1 \\ 
\cdots & -1  & -1  &  0  & -1  & -1  &  0  &  0  & -1  & -1  &  0  & -1 \\ 
\cdots &  0  & -1  & -1  &  0  &  0  & -1  & -1  &  0  &  1  & -1  &  1 \\ 
\cdots & -1  & -1  & -1  &  0  & -1  & -1  &  0  &  0  &  0  & -1  &  0 \\ 
\cdots &  0  &  0  & -1  &  0  &  0  &  0  & -1  &  1  &  0  &  0  &  1 \\ 
\cdots &  0  &  0  &  0  &  0  &  0  &  0  &  0  &  0  &  0  &  0  &  0  
\end{array}
\right]
\end{gather*}

\begin{gather*}
T=8, \\
\left[
\begin{array}{rrrrrrrrrrrrrrrr}
 -7  & -4  & -3  & -3  & -2  & -2  &  7  &  7  &  5  &  5  &  4  &  4  &  4  &  3  &  3  &  1 \\ 
 -7  & -4  & -4  & -2  & -3  & -3  &  8  &  7  &  5  &  5  &  5  &  4  &  3  &  4  &  4  &  0 \\ 
 -6  & -3  & -3  & -2  & -2  & -2  &  7  &  7  &  4  &  4  &  5  &  3  &  3  &  4  &  3  &  0 \\ 
  1  &  1  &  0  &  1  &  0  &  0  &  0  & -1  & -1  &  0  &  0  &  0  & -1  &  0  &  0  & -1 \\ 
  2  &  1  &  1  &  1  &  1  &  0  & -1  & -1  & -1  & -1  &  0  & -1  & -1  &  0  &  0  & -1 \\ 
  1  &  1  &  1  &  0  &  1  &  1  & -2  & -1  &  0  & -1  & -1  &  0  &  0  & -1  & -1  &  1 \\ 
  2  &  1  &  1  &  1  &  1  &  1  & -2  & -1  & -1  & -1  & -1  & -1  &  0  &  0  & -1  &  0 \\ 
  0  &  0  &  1  &  0  &  0  &  1  & -1  & -1  &  0  &  0  & -1  &  0  &  0  & -1  &  0  &  1 \\ 
  0  &  0  &  0  &  0  &  0  &  0  &  0  &  0  &  0  &  0  &  0  &  0  &  0  &  0  &  0  &  0  
\end{array}
\right]
\end{gather*}

\begin{gather*}
T=9, \\
\left[
\begin{array}{rrrrrrrrrrrrr}
 12  & -8  & -5  & -4  & -4  & -4  & -3  & -3  & -2  &  8  &  8  &  6  & \cdots \\
 12  & -8  & -4  & -5  & -4  & -4  & -3  & -3  & -3  &  9  &  8  &  5  & \cdots \\
 11  & -7  & -4  & -4  & -4  & -4  & -2  & -2  & -3  &  8  &  8  &  5  & \cdots \\
 -1  &  1  &  1  &  1  &  1  &  0  &  1  &  0  &  0  &  0  & -1  & -1  & \cdots \\
 -1  &  2  &  2  &  1  &  1  &  1  &  1  &  1  &  0  & -1  & -1  & -1  & \cdots \\
 -2  &  1  &  1  &  1  &  1  &  1  &  0  &  1  &  1  & -2  & -1  &  0  & \cdots \\
 -2  &  2  &  1  &  2  &  1  &  1  &  1  &  1  &  1  & -2  & -1  & -1  & \cdots \\
 -1  &  0  &  0  &  0  &  0  &  1  &  0  &  0  &  1  & -1  & -1  &  0  & \cdots \\
  0  &  0  &  0  &  0  &  0  &  0  &  0  &  0  &  0  &  0  &  0  &  0  & \cdots
\end{array} \right. \\
\left. \qquad \qquad \begin{array}{rrrrrrrrrrrr}
\cdots &   5  &  5  &  4  &  4  &  4  &  4  &  4  &  4  &  3  &  1 \\ 
\cdots &   6  &  6  &  4  &  4  &  4  &  4  &  4  &  3  &  4  &  0 \\ 
\cdots &   5  &  5  &  4  &  4  &  4  &  4  &  3  &  4  &  4  &  0 \\ 
\cdots &   0  &  0  & -1  &  0  &  0  &  0  & -1  & -1  &  1  & -1 \\ 
\cdots &  -1  &  0  & -1  & -1  &  0  &  0  & -1  & -1  &  0  & -1 \\ 
\cdots &  -1  & -1  &  0  &  0  & -1  & -1  &  0  &  1  & -1  &  1 \\ 
\cdots &  -1  & -1  &  0  & -1  & -1  &  0  &  0  &  0  & -1  &  0 \\ 
\cdots &   0  & -1  &  0  &  0  &  0  & -1  &  1  &  0  &  0  &  1 \\ 
\cdots &   0  &  0  &  0  &  0  &  0  &  0  &  0  &  0  &  0  &  0  
\end{array}
\right]
\end{gather*}

\bibliographystyle{plain}
\bibliography{references}

\begin{thebibliography}{10}

\bibitem{4ti2}
4ti2 team.
\newblock 4ti2---a software package for algebraic, geometric and combinatorial
  problems on linear spaces.
\newblock {A}vailable at www.4ti2.de.

\bibitem{normaliz}
Winfried Bruns, Bogdan Ichim, and Christof S\"oger.
\newblock {\tt Normaliz}, a tool for computations in affine monoids, vector
  configurations, lattice polytopes, and rational cones, 2011.

\bibitem{diaconis-sturmfels}
Persi Diaconis and Bernd Sturmfels.
\newblock Algebraic algorithms for sampling from conditional distributions.
\newblock {\em The Annals of Statistics}, 26(1):363--397, 1998.

\bibitem{polymake}
Ewgenij Gawrilow and Michael Joswig.
\newblock polymake: a framework for analyzing convex polytopes.
\newblock In Gil Kalai and G\"unter~M. Ziegler, editors, {\em Polytopes ---
  Combinatorics and Computation}, pages 43--74. Birkh\"auser, 2000.

\bibitem{Hara:2010vn}
Hisayuki Hara and Akimichi Takemura.
\newblock Markov chain monte carlo test of toric homogeneous markov chains,
  2010.

\bibitem{Hara:2010uq}
Hisayuki Hara and Akimichi Takemura.
\newblock A markov basis for two-state toric homogeneous markov chain model
  without initial paramaters.
\newblock {\em Journal of Japan Statistical Society}, 41, 2011.

\bibitem{MS2005}
Ezra Miller and Bernd Sturmfels.
\newblock {\em Combinatorial commutative algebra}.
\newblock Graduate texts in mathematics. Springer, 2005.

\bibitem{Pachter:2005kx}
Lior Pachter and Bernd Sturmfels.
\newblock {\em Algebraic Statistics for Computational Biology.}
\newblock Cambridge University Press, Cambridge, UK, 2005.

\bibitem{Schrijver1986Theory-of-linea}
Alexandeer Schrijver.
\newblock {\em Theory of Linear and Integer Programming}.
\newblock John Wiley \& Sons, Inc., New York, NY, USA, 1986.

\bibitem{sturmfels1996}
Bernd Sturmfels.
\newblock {\em Gr\"obner Bases and Convex Polytopes}, volume~8 of {\em
  University Lecture Series}.
\newblock American Mathematical Society, Providence, RI, 1996.

\bibitem{yap2000}
Chee-Keng Yap.
\newblock {\em Fundamental problems of algorithmic algebra}.
\newblock Oxford University Press, 2000.

\end{thebibliography}

\end{document}